\RequirePackage{fix-cm}
\documentclass{svjour3}                     
\smartqed  
\usepackage{mathptmx}      
\usepackage{amsmath}
\usepackage{amssymb}

\usepackage[numbers]{natbib}

\usepackage[english]{babel}
\usepackage[utf8]{inputenc}

\usepackage{subfig}
\usepackage{graphicx}

\usepackage[unicode]{hyperref}
\usepackage[active]{srcltx} 

\usepackage{mathbbol}
\usepackage{bm} 

\usepackage{units}
\usepackage{tensor}
\usepackage{accents}

\DeclareMathOperator{\divergence}{div}

\newcommand{\exponential}[1]{\ensuremath{{\mathrm e}^{#1}}}

\newcommand{\reference}{\mathrm{ref}}
\newcommand{\crit}{\mathrm{crit}}
\newcommand{\bydefinition}{\mathrm{def}}

\newcommand{\diff}{\mathrm{d}}
\newcommand{\Diff}[1][]{\mathrm{D}_{#1}} 

\renewcommand{\vec}[1]{\ensuremath{\mathbf{#1}}}

\makeatletter
\@ifpackageloaded{bm}%
{\renewcommand{\vec}[1]{\ensuremath{\bm{#1}}}%
}{%
\relax
}
\makeatother
\newcommand{\tensorq}[1]{\ensuremath{\mathbb{#1}}}      

\newcommand{\transpose}[1]{#1^\top}

\newcommand{\identity}{\ensuremath{\tensorq{I}}}

\newcommand{\cstress}{\tensorq{T}}

\makeatletter
\@ifpackageloaded{bm}%
{%
}{%

}

\@ifpackageloaded{bm}%
{%
 
}{%

}

\@ifpackageloaded{bm}%
{%
}{%

}
\makeatother

\newcommand{\vecv}{\ensuremath{\vec{v}}}
\newcommand{\gradv}{\ensuremath{\nabla \vecv}}

\newcommand{\gradsym}{\ensuremath{\tensorq{D}}}

\newcommand{\R}{\ensuremath{{\mathbb R}}}

\newcommand{\N}{\ensuremath{{\mathbb N}}}

\newcommand{\ienergy}{\ensuremath{e}} 
\newcommand{\fenergy}{\ensuremath{\psi}} 
\newcommand{\entropy}{\ensuremath{\eta}} 

\newcommand{\temp}{\ensuremath{\theta}} 

\newcommand{\nettenergy}{\ensuremath{E}_{\mathrm{tot}}} 
\newcommand{\netentropy}{\ensuremath{S}} 

\newcommand{\cheatvol}{\ensuremath{c_{\mathrm{V}}}}
\newcommand{\cheatvolref}{\ensuremath{c_{\mathrm{V}, \reference}}} 

\newcommand{\hfluxc}{\vec{j}_{q}}     

\newcommand{\entfluxc}{\vec{j}_{\entropy}} 

\newcommand{\entprodctemp}{\zeta} 

\newcommand{\pd}[2]{\ensuremath{\frac{\partial {#1}}{\partial {#2}}}}
\newcommand{\ppd}[2]{\ensuremath{\frac{\partial^2 {#1}}{\partial {#2^2}}}}
\newcommand{\dd}[2]{\ensuremath{\frac{\diff {#1}}{\diff {#2}}}}

\makeatletter
\@ifpackageloaded{tensor}
{

}{%

}
\makeatother

\makeatletter
\@ifpackageloaded{tensor}
{

}{%

}
\makeatother

\newcommand{\norm}[2][]{\ensuremath{\left\|#2\right\|_{#1}}}
\newcommand{\absnorm}[1]{\ensuremath{\left|#1\right|}}

\makeatletter
\@ifundefined{volume}{%
\newcommand{\volume}[1][\Omega]{\ensuremath{#1}}}%
{%
\renewcommand{\volume}[1][\Omega]{\ensuremath{#1}}}
\makeatother

\newcommand{\cvolumee}{\diff \mathrm{v}}

\newcommand{\csurfacees}{\diff \mathrm{s}}

\makeatletter
\@ifpackageloaded{MnSymbol} 
{
\newcommand{\tensordot}[2]{\ensuremath{#1 \vdotdot #2}} 
}{%
\newcommand{\tensordot}[2]{\ensuremath{#1 : #2}} 
}
\makeatother
\newcommand{\vectordot}[2]{\ensuremath{#1 \bullet #2}}

\newcommand{\sleb}[2]{\ensuremath{L}^{#1} \left(#2 \right)}             

\newcommand{\ssob}[3]{\ensuremath{W}^{#1, #2} \left(#3 \right)}         
\newcommand{\ssobzero}[3]{\ensuremath{W}_{0}^{#1, #2} \left(#3 \right)} 

\newcommand{\entropyrellp}{{\mathcal S}}
\newcommand{\energyrellp}{{\mathcal E}}

\newcommand{\kapparef}{\kappa_{\reference}}
\newcommand{\tempref}{\temp_{\reference}}

\newcommand{\boundary}{\mathrm{bdr}}
\newcommand{\tempbdr}{\temp_{\boundary}}
\newcommand{\initial}{\mathrm{init}}

\newcommand{\Hdiff}{\relentropy}
\newcommand{\coeffC}{C_{\Hdiff}}

\newcommand{\vartemp}{\vartheta}
\newcommand{\vartempref}{\vartheta_{\reference}}

\newcommand{\relentropy}{\entropy_{\mathrm{diff}}}

\numberwithin{equation}{section}
\let\cite\citet

\begin{document}

\title{Unconditional finite amplitude stability of a fluid in a mechanically isolated vessel with spatially non-uniform wall temperature\thanks{V\'{\i}t Pr\r{u}\v{s}a thanks the Czech Science Foundation, grant number 20-11027X, for its support. Mark Dostal\'{\i}k has been supported by Charles University Research program No. UNCE/SCI/023 and GAUK 1652119.} 
}

\titlerunning{Finite amplitude stability of incompressible heat conducting viscous fluid}        

\author{M. Dostal\'{\i}k \and V. Pr\r{u}\v{s}a \and K.~R.~Rajagopal}


\institute{M. Dostal\'{\i}k and V. Pr\r{u}\v{s}a \at
  Charles University, Faculty of Mathematics and Physics, Sokolovsk\'a 83, Praha 8 -- Karl\'{\i}n\\
  CZ 186 75, Czech Republic\\
\email{dostalik@karlin.mff.cuni.cz}, \email{prusv@karlin.mff.cuni.cz}
\and
K. R. Rajagopal \at
Texas A\&M University, Department of Mechanical Engineering, 3123 TAMU, College Station\\
TX 77843-3123, United States of America\\
\email{krajagopal@tamu.edu}
}

\date{Received: date / Accepted: date}

\maketitle

\begin{abstract}
   A fluid occupying a mechanically isolated vessel with walls kept at spatially non-uniform temperature is in the long run expected to reach the spatially inhomogeneous steady state. Irrespective of the initial conditions the velocity field is expected to vanish, and the temperature field is expected to be fully determined by the steady heat equation. This simple observation is however difficult to prove using the corresponding governing equations. The main difficulties are the presence of the dissipative heating term in the evolution equation for temperature and the lack of control on the heat fluxes through the boundary. Using thermodynamically based arguments, it is shown that these difficulties in the proof can be overcome, and it is proved that the velocity and temperature perturbations to the steady state actually vanish as the time goes to infinity. 
 
\end{abstract}

\section{Introduction}
\label{sec:introduction}
The everyday experience is that if a fluid is put into a vessel, and if it is not allowed to substantially interact with the outside environment, then it eventually comes to the rest state. Moreover, the rest state is attained irrespective of the initial state of the fluid. The question is whether this behaviour can be deduced using the corresponding governing equations such as the Navier--Stokes--Fourier equations for an incompressible fluid.

In the case of a vessel that is \emph{completely isolated} from the outside environment, see Figure~\ref{fig:systems-isolated}, the answer to this question is straightforward and positive. One can exploit basic ideas from continuum thermodynamics, see~\cite{coleman.bd:on}, \cite{gurtin.me:thermodynamics*1,gurtin.me:thermodynamics} and early considerations by~\cite{duhem.p:traite}, and prove that both the velocity and temperature field eventually reach the \emph{spatially homogeneous} equilibrium rest state. The same holds also for \emph{simple open systems} such as a vessel immersed in a thermal bath, that is for the mechanically isolated vessel with walls kept at constant spatially uniform temperature, see~Figure~\ref{fig:systems-bath}. In a slightly more general case the answer to the question is much more complicated.

\begin{figure}[t]
  \begin{center}
  \subfloat[\label{fig:systems-isolated}Isolated vessel.]{\includegraphics[width=0.30\textwidth]{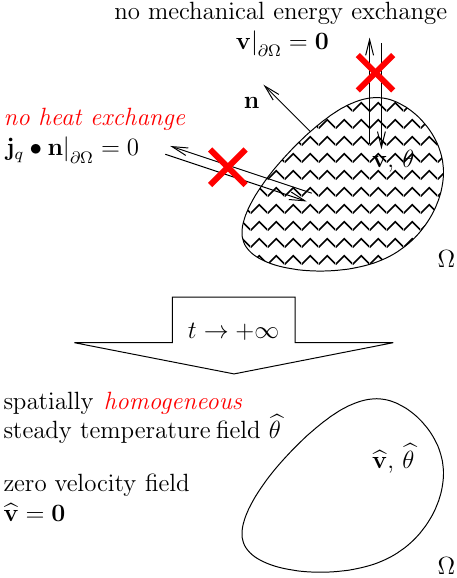}}
  \qquad
  \subfloat[\label{fig:systems-bath}Thermal bath.]{\includegraphics[width=0.30\textwidth]{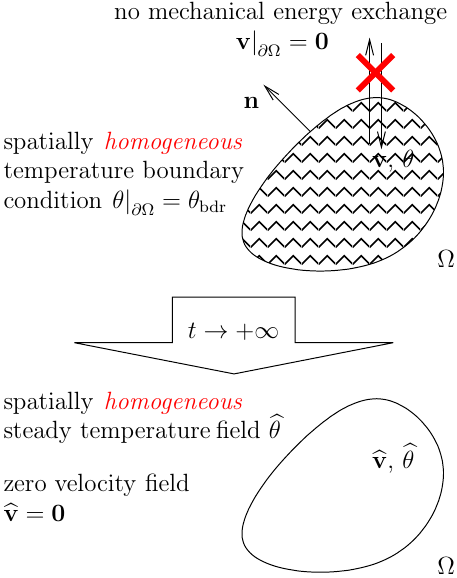}}
  \qquad
  \subfloat[\label{fig:systems-walls}Spatially non-uniform wall temperature.]{\includegraphics[width=0.3\textwidth]{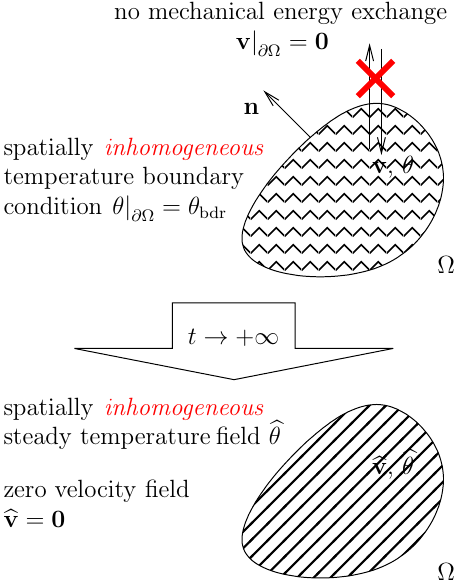}}
  \end{center}
  \label{fig:systems}
  \caption{Long time behaviour under various types of temperature boundary conditions.}
\end{figure}

If we consider a \emph{mechanically isolated} vessel with \emph{walls kept at a given spatially non-uniform temperature}, see Figure~\ref{fig:systems-walls}, that is if we change the temperature boundary condition, we are not anymore dealing with a thermodynamically isolated system or a system immersed in a thermal bath, and the concepts applicable in the stability analysis of these systems are of no use. Indeed, the \emph{change of temperature boundary condition}, which on the formal level means the change from the zero Neumann boundary condition/constant Dirichlet boundary condition to the spatially inhomogeneous Dirichlet boundary condition, qualitatively changes both the steady state and the mode of interaction with the outside environment. 

However, even in this setting we again expect the fluid to reach a steady state. The expected steady state is the state with the zero velocity field, while the steady temperature field is given by the steady heat equation and it is~\emph{spatially inhomogeneous}. This means that from the thermodynamical perspective we are dealing with a \emph{non-equilibrium} (entropy producing) steady state.

Despite the increased complexity of the steady state, it is still straightforward to prove that \emph{velocity field} decays to zero, see for example~\cite{serrin.j:on}. This reflects the fact that the kinetic energy is spontaneously converted to the thermal energy via the dissipative heating, hence if there is no mechanical energy supply from outside, then we should observe the decay of the kinetic energy.

Concerning the \emph{temperature} field, the situation is different. First, if the fluid is allowed to exchange heat with its surrounding, then it is \emph{a priori} not clear, whether one will during the evolution of the system see heat~\emph{influx} or \emph{efflux}. Second, the thermal energy is, in our simple case of incompressible Navier--Stokes--Fourier fluid, not converted to another type of energy. It is \emph{merely rearranged in space}. This is in striking difference with the behaviour of the kinetic energy that is being dissipated and that naturally ``disappears'' from the system.

Consequently, the key question \emph{whether the governing equations actually predict the expected long time behaviour}, becomes in the case of spatially inhomogeneous wall temperature rather difficult to answer. In particular, any characterisation of the long time behaviour of \emph{temperature perturbations} seems to be difficult to obtain in such a setting. In fact, to the best of our knowledge, no proof regarding the decay of temperature perturbations is available in the literature so far.

\subsection{Main result}
\label{sec:main-result}
We investigate the key question for the fluid described by the standard incompressible Navier--Stokes--Fourier model, and we provide a positive answer to the key question. (We consider the motion in the absence of external forces and we assume that the fluid density is constant. No Oberbeck--Boussinesq approximation is considered.) Although we consider stability of \emph{spatially inhomogeneous} non-equilibrium steady state in a thermodynamically \emph{open} system, we are able during the proof to exploit thermodynamical ideas, though these ideas go beyond the concepts introduced by~\cite{coleman.bd:on} in the stability analysis of \emph{spatially homogeneous} steady states in isolated systems or systems immersed in a thermal bath.

In particular, \emph{assuming that the solution to the governing equations is the classical one}, we prove that
\begin{equation}
  \label{eq:1}
  \int_{\Omega}
  \rho
  \cheatvolref
  \widehat{\temp}
  \left[
    \frac{1}{n}
    \left(
      1 + \frac{\widetilde{\temp}}{\widehat{\temp}}
    \right)^{n}
    -
    \frac{1}{m}
    \left(
      1 + \frac{\widetilde{\temp}}{\widehat{\temp}}
    \right)^{m}
    +
    \frac{n-m}{mn}
  \right]
  \,
  \cvolumee
  \xrightarrow{t \to + \infty}
  0
  ,
\end{equation}
where $m,n \in (0,1)$, $n>m > \frac{n}{2}$, see Theorem~\ref{thr:1} and Corollary~\ref{crl:1}. Here $\widehat{\temp}$ denotes the temperature field in the spatially inhomogeneous non-equilibrium steady state (solution to the steady heat equation), and $\widetilde{\temp}$ denotes the perturbation with respect to~$\widehat{\temp}$, and the symbols~$\rho$ and $\cheatvolref$ denote the density and the specific heat capacity at constant volume. The functional in~\eqref{eq:1} is non-negative and it vanishes if and only if the perturbation $\widetilde{\temp}$ vanishes everywhere in the domain~$\Omega$ of interest, which means that~\eqref{eq:1} is the desired stability result. \emph{We note that \eqref{eq:1} holds irrespective of the initial state of the system}, and with \emph{minimal assumptions concerning the behaviour of the dissipative heating term} in the evolution equation for the temperature.

\enlargethispage*{1em}

\subsection{Main issues}
\label{sec:main-issues}
Before we proceed with the proof of~\eqref{eq:1}, we briefly comment on the main issues in the analysis of long time behaviour of the temperature perturbations. (See Section~\ref{sec:preliminaries} for the notation and precise formulation of the long time behaviour problem.)

The first issue is the presence of the \emph{dissipative heating term} $2 \mu \tensordot{\gradsym}{\gradsym}$ in the evolution equation for the temperature. If we were following the so-called energy method, see~\cite{joseph.dd:stability*1,joseph.dd:stability} or \cite{straughan.b:energy}, we would have investigated the behaviour of the temperature perturbations $\widetilde{\temp}$ by the means of functional
$
  \int_{\Omega} \widetilde{\temp}^2 \, \cvolumee
$.
The evolution equation for this functional is easy to obtain from the evolution equation for the perturbation~$\widetilde{\temp}$. It suffices to take the product of the evolution equation for the perturbation with the perturbation itself and integrate \emph{by parts}, and one arrives at the equation
\begin{equation}
  \label{eq:3}
  \rho
  \cheatvolref
  \dd{}{t}
  \int_{\Omega} \widetilde{\temp}^2 \, \cvolumee
  =
  -
  \int_{\Omega}
  \kapparef \vectordot{\nabla \widetilde{\temp}}{\nabla \widetilde{\temp}}
  \, \cvolumee
  +
  \int_{\Omega}
  2 \mu \tensordot{\widetilde{\gradsym}}{\widetilde{\gradsym}}
  \;
  \widetilde{\temp}
  \, \cvolumee
  +
  \int_{\Omega}
  \rho
  \cheatvolref
  \left(
    \vectordot{\widetilde{\vec{v}}}{\nabla \widehat{\temp}}
  \right)
  \widetilde{\temp}
  \, \cvolumee
  .
\end{equation}

The first term on the right-hand side is a favourable one. It is negative and it pushes the temperature perturbation to zero. On the other hand the second term that comes from the dissipative heating does not have a sign, and it is hard to control/estimate unless one is willing to search for further information\footnote{Further information on the product $\tensordot{\widetilde{\gradsym}}{\widetilde{\gradsym}}$ can be formally obtained by the multiplication of the evolution equation~\eqref{eq:34} by the Laplacian of the velocity perturbation, see~\cite{kagei.y.ruzicka.m.ea:natural} for a similar manipulation, and also the discussion of slightly compressible convection in~\cite{straughan.b:energy} and \cite{richardson.ll:nonlinear}. This approach would however work only for constant viscosity, and it would allow one to prove stability only for a restricted set of initial perturbations (small perturbations). We however aim at unconditional result -- the size of initial perturbation must not be limited. This is what we expect intuitively from our physical system.} concerning the behaviour of $2 \mu \tensordot{\widetilde{\gradsym}}{\widetilde{\gradsym}}$. In particular, even if we know that the velocity perturbation $\widetilde{\vec{v}}$ decays, which is true, the fact that the velocity perturbation $\widetilde{\vec{v}}$ is small does not \emph{a~priori} imply that the velocity gradient $\widetilde{\gradsym}$ is also small. (A rapidly oscillating function can be small, while its gradient can be arbitrarily large.) Further, one can not expect to obtain an explicit pointwise \emph{upper bound} on the temperature perturbation $\widetilde{\temp}$. The source term $2 \mu \tensordot{\widetilde{\gradsym}}{\widetilde{\gradsym}}$ in the heat equation may concentrate at a spatial point, and push the local value of the temperature above any \emph{a~priori} given threshold.

In fact, the only piece of information one is allowed to use in the analysis of the evolution of temperature perturbations is
\begin{equation}
  \label{eq:4}
  \int_{\tau=0}^{+\infty}
  \left(
    \int_{\Omega}
    2 \mu \tensordot{\widetilde{\gradsym}}{\widetilde{\gradsym}}
    \, \cvolumee
  \right)
  \,
  \diff \tau
  <
  +\infty
  ,
\end{equation}
which means that the heat released by the dissipation of the kinetic energy is finite. However, it is not known \emph{when} and \emph{where} the heat is released. (We note that if the integral $\int_{\tau=0}^{\infty} g(\tau) \, \diff \tau $ of a positive function $g$ is finite, it is not necessarily true that $g \to 0+$ as $\tau \to +\infty$. See also the discussion following Lemma~\ref{lm:10} below.) In other words, the source term in the temperature evolution equation~\eqref{eq:35} could be triggered at \emph{a~priori} unknown time instants and spatial locations. 

The second issue is that the heat flux through the boundary is in the given setting beyond our control. We do not know \emph{a priori} whether at the given time instant and spatial location the heat flows \emph{into} the domain of interest or \emph{out of} the domain of interest.

For these reasons the dissipative heating is conveniently neglected in most mathematical studies on the stability of fluid motion, see~\cite{joseph.dd:stability*1,joseph.dd:stability} or \cite{straughan.b:energy}. 

\subsection{Outline of the solution to the stability problem}
\label{sec:outl-solut-stab}
If the only piece of information we are willing to use is~\eqref{eq:4}, we can hardly expect any result concerning the~\emph{decay rate} of the perturbations. Consequently, if we ignore the spatial dependence, we need to work with very weak qualitative results concerning the asymptotic behaviour of time dependent functions. In our analysis we use the following lemma.
\enlargethispage{1.5em}

\begin{lemma}[Decay of integrable functions]
  \label{lm:1}
  \begin{subequations}
    \label{eq:5}
  Let $y: [0, +\infty) \mapsto \R^+$ be a continuous non-negative function such that
  \begin{equation}
    \label{eq:6}
    \int_{\tau=0}^{+\infty} y(\tau) \, \diff \tau \leq C_1,
  \end{equation}
  where $C_1$ is a constant. Moreover, let for all $s, t \in [0, + \infty)$, $t>s$, 
  \begin{equation}
    \label{eq:7}
    y(t) - y(s) \leq \int_{\tau=s}^t f(y(\tau)) \, \diff \tau +  \int_{\tau=s}^t h(\tau) \, \diff \tau
  \end{equation}
  hold, where $f$ is a nondecreasing function from $\R^+$ to $\R^+$ and $h$ is a non-negative function such that
  $
  \int_{\tau=0}^{+\infty} h(\tau) \, \diff \tau \leq C_2
  $,
  where $C_2$ is a constant. Then
  \begin{equation}
    \label{eq:9}
    \lim_{t \to + \infty} y(t) = 0.
  \end{equation}
  \end{subequations}
\end{lemma}
The lemma is taken from~\cite[Lemma 1.3]{zheng.s:asymptotic}, and it originated in the work of~\cite[Theorem 9]{krejc.p.sprekels.j:weak}, see also~\cite{zheng.s:nonlinear} for further discussion. Various generalisations of the lemma can be found in~\cite{ramm.ag.hoang.ns:dynamical}, and we also point out that the question on decay of integrable functions is also a subject of {B}arb\u{a}lat lemma known in the optimal control theory, see~\cite{farkas.b.wegner.s:variations}. Lemma~\ref{lm:1} can be also rewritten in the differential form, but in this case one has to assume the existence of the derivative. This modification of Lemma~\ref{lm:1} reads as follows, see~\cite[Lemma 1.3]{zheng.s:asymptotic}.

\begin{lemma}[Decay of integrable functions -- differential version]
  \label{lm:10}
    Let $y: [0, +\infty) \mapsto \R^+$ be a continuous non-negative function such that
  \begin{equation}
    \label{eq:29}
    \int_{\tau=0}^{+\infty} y(\tau) \, \diff \tau \leq C_1,
  \end{equation}
  where $C_1$ is a constant, and let us assume that $\dd{y}{t}$ is a locally integrable function on $\R^+$. Suppose that $y$ satisfies for all $t \in (0, + \infty)$ the condition
  \begin{equation}
    \label{eq:21}
    \dd{y}{t} \leq f(y) + h(t),
  \end{equation}
  where functions $f$ and $h$ have the same properties as in Lemma~\ref{lm:1}. Then
  \begin{equation}
    \label{eq:25}
    \lim_{t \to + \infty} y(t) = 0.
  \end{equation}
\end{lemma}
In our case the existence of the time derivative of the corresponding quantity $y$ will be granted since we assume that we work with the classical solution of the governing equations. Consequently we can freely switch---whenever convenient---between the integral formulation used in Lemma~\ref{lm:1} and the differential formulation used in Lemma~\ref{lm:10}.

We recall the well-known fact that the existence of the integral~\eqref{eq:6} does not on its own imply~\eqref{eq:9}, the classical counterexample is the function
\begin{equation}
  \label{eq:26}
  g(t)
  =
  \begin{cases}
    1, & t \in [N-\frac{1}{2N^2}, N+\frac{1}{2N^2}],\  N \in \N, \\
    0, & \text{otherwise},
  \end{cases}
\end{equation}
see Figure~\ref{fig:fceint} for a sketch of such a function. (If necessary, using the convolution with the standard smoothing kernel, the function can be made infinitely smooth.) For this function we have
\begin{equation}
  \label{eq:28}
  \int_{t=0}^{+\infty} g(t) \, \diff t = \sum_{N=1}^{+\infty} \frac{1}{N^2} < + \infty,
\end{equation}
but $\lim_{N \to + \infty} g(N) = 1$. Consequently we see that
\begin{equation}
  \label{eq:27}
  \lim_{t \to + \infty} g(t) \not = 0, 
\end{equation}
and the limit in fact does not exist at all. The condition~\eqref{eq:7} or~\eqref{eq:21} in Lemma~\ref{lm:1} and Lemma~\ref{lm:10} respectively effectively prohibits the existence of the sufficiently fast thinning peaks as that shown in Figure~\ref{fig:fceint}. (The maximum slope of the graph that is the derivative is limited.) This provides an intuitive argument for the validity of the Lemma~\ref{lm:1} and Lemma~\ref{lm:10}. The rigorous proof of the lemmas is given for example in~\cite{zheng.s:nonlinear}. 
\begin{figure}[h]
  \begin{center}
  \includegraphics[width=0.5\textwidth]{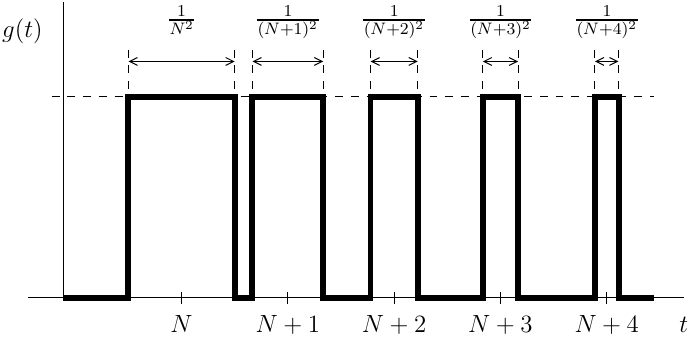}
  \end{center}
  \caption{Function $g(t)$ -- integrable function that does not vanish at infinity. (Vertical segments are not part of the graph, they have been added for better visualisation.)}
  \label{fig:fceint}
\end{figure}

Once we decide that Lemma~\ref{lm:1}/Lemma~\ref{lm:10} or its generalisation is the right tool for stability analysis, the only remaining task is to find the quantity $y(t)$ that satisfies the assumptions of the lemma, and that vanishes if and only if the temperature perturbation vanishes.

\emph{First}, we try to solve this problem using a Lyapunov type functional constructed via the method proposed by~\cite{bulcek.m.malek.j.ea:thermodynamics}. The functional constructed by the method reads
\begin{subequations}

\begin{equation}
  \label{eq:10}
  \mathcal{V}_{\mathrm{meq}}
  \left(
    \left.
      \widetilde{\vec{W}}
    \right\|
    \widehat{\vec{W}}
  \right)
  =_{\bydefinition}
  \int_{\Omega}
  \rho
  \cheatvolref
  \widehat{\temp}
  \left[
    \frac{\widetilde{\temp}}{\widehat{\temp}}
    -
    \ln \left( 1 + \frac{\widetilde{\temp}}{\widehat{\temp}} \right)
  \right]
  \,
  \cvolumee
  +
  \int_{\Omega}
  \frac{1}{2} \rho \absnorm{\widetilde{\vec{v}}}^2
  \,
  \cvolumee
  ,
\end{equation}
and its time derivative reads
\begin{multline}
  \label{eq:11}
  \dd{}{t}
    \mathcal{V}_{\mathrm{meq}}
    \left(
      \left.
        \widetilde{\vec{W}}
      \right\|
      \widehat{\vec{W}}
    \right)
    =
    -
    \int_{\Omega}
    \kapparef
    \widehat{\temp}
    \vectordot{\nabla \ln \left(1 + \frac{\widetilde{\temp}}{\widehat{\temp}} \right)}{\nabla \ln \left(1 + \frac{\widetilde{\temp}}{\widehat{\temp}} \right)}
    \,
    \cvolumee
    -
    \int_{\Omega}
    \frac{2 \mu \tensordot{\widetilde{\gradsym}}{\widetilde{\gradsym}}}{1 + \frac{\widetilde{\temp}}{\widehat{\temp}}}
    \,
    \cvolumee
    \\
    +
    \int_{\Omega}
    \rho
    \cheatvolref
    \left(
      \vectordot{\nabla \widehat{\temp}}{\widetilde{\vec{v}}}
    \right)
    \ln \left(1 + \frac{\widetilde{\temp}}{\widehat{\temp}} \right)
    \,
    \cvolumee
    ,
  \end{multline}
\end{subequations}
see Section~\ref{sec:lyap-like-funct}. We see that the first two terms on the right-hand side are negative, while last term can be positive or negative depending on the actual evolution of the perturbation. Consequently, it is not clear whether the time derivative of the proposed functional is negative. Despite these complications, the functional $\mathcal{V}_{\mathrm{meq}}$ is a promising one for the study of the stability problem, and it is almost the right one. We also note, that if we were dealing with a spatially homogeneous steady state $\widehat{\temp}$, which is the steady state in a thermodynamically isolated vessel or in a vessel immersed in a thermal bath, then $\nabla \widehat{\temp} = 0$, and the last term would vanish, and the stability problem would be solved immediately.

\emph{Second}, in Section~\ref{sec:reth-devel-lyap} we revisit the construction of the prospective Lyapunov type functional $\mathcal{V}_{\mathrm{meq}}$. We exploit the idea that the stability of the rest state should be an absolute fact independent of the choice of the description of the system. In particular, it should be independent of the choice of the temperature scale. In~\eqref{eq:10} we however use the \emph{absolute temperature} which is just a convenient measure of hotness. But it is not necessary to use the absolute temperature scale, other temperature scales can be used as well.

For the alternative temperature scale, we use the scale
$
  \frac{\vartemp}{\vartempref} =_{\bydefinition} \left( \frac{\temp}{\tempref} \right)^{1-m}
$,
where $\vartemp$ is the alternative temperature scale and $\vartempref$ and $\tempref$ are some fixed reference temperatures and $m \in (0,1)$. It is straightforward to check that the alternative temperature scale preserves the ordering according to the hotness, which means that the heat still flows in the direction of the temperature difference and so forth, but quantitative characterisation of the flow is different.

With respect to this alternative choice of temperature scale, the fluid of interest virtually behaves like a fluid with \emph{a temperature dependent thermal conductivity and temperature dependent specific heat capacity at constant volume}. For this virtual fluid we can again use the method proposed by~\cite{bulcek.m.malek.j.ea:thermodynamics} and construct a candidate for a Lyapunov type functional~$\mathcal{V}_{\mathrm{meq}}^{\vartemp,\, m}$. If we rewrite the functional $\mathcal{V}_{\mathrm{meq}}^{\vartemp,\, m}$ using the original temperature scale~$\temp$ it reads
  \begin{equation}
    \label{eq:14}
    \mathcal{V}_{\mathrm{meq}}^{\vartemp,\, m}
    \left(
      \left.
        \widetilde{\vec{W}}
      \right\|
      \widehat{\vec{W}}
    \right)
    =_{\bydefinition}
    \int_{\Omega}
    \rho
    \cheatvolref
    \widehat{\temp}
    \left[
      \frac{\widetilde{\temp}}{\widehat{\temp}}
      -
      \frac{1}{m}
      \left(
        \left(
          1 + \frac{\widetilde{\temp}}{\widehat{\temp}}
        \right)^{m}
        -
        1
      \right)
    \right]
    \,
    \cvolumee
    +
    \int_{\Omega}
    \frac{1}{2} \rho \absnorm{\widetilde{\vec{v}}}^2
    \,
    \cvolumee   
    ,
  \end{equation}
and its time derivative is again given by an explicit formula. Using this manipulation we obtain a whole class of functionals that are again almost appropriate for the stability analysis.

\emph{Third}, we conclude our analysis by exploiting Lemma~\ref{lm:1}. Using the family of functionals~\eqref{eq:14} we derive, see Section~\ref{sec:appl-lemma-decay}, the inequality
\begin{equation}
  \label{eq:16}
  \dd{}{t}
  \left(
    \mathcal{V}_{\mathrm{meq}}^{\vartemp,\, m}
    -
    \mathcal{V}_{\mathrm{meq}}^{\vartemp, n}
  \right)
  \leq
  -
  K^{m,n}
  \left(
    \mathcal{V}_{\mathrm{meq}}^{\vartemp,\, m}
    -
    \mathcal{V}_{\mathrm{meq}}^{\vartemp, n}
  \right)
  +
  \mathcal{H}^{m,n}
  ,
\end{equation}
where $K^{m,n}$ is a positive constant, $\mathcal{H}^{m,n}$ is a non-negative function integrable with respect to time, and $m, n \in (0,1)$, $n>m>\frac{n}{2}$. In~\eqref{eq:16} it is straightforward to show that the quantity
\begin{equation}
  \label{eq:17}
  \mathcal{Y}^{m,n}
  =_{\bydefinition}
  \mathcal{V}_{\mathrm{meq}}^{\vartemp,\, m}
  -
  \mathcal{V}_{\mathrm{meq}}^{\vartemp, n}
  =
  \int_{\Omega}
  \rho
  \cheatvolref
  \widehat{\temp}
  \left[
    \frac{1}{n}
    \left(
      1 + \frac{\widetilde{\temp}}{\widehat{\temp}}
    \right)^{n}
    -
    \frac{1}{m}
    \left(
      1 + \frac{\widetilde{\temp}}{\widehat{\temp}}
    \right)^{m}
    +
    \frac{n-m}{mn}
  \right]
  \,
  \cvolumee
  ,
\end{equation}
is a \emph{non-negative quantity that vanishes if and only if the temperature perturbation $\widetilde{\temp}$ vanishes}. Using~\eqref{eq:16} we can conclude that 
$
\int_{\tau=0}^{+\infty}
  \mathcal{Y}^{m,n}(\tau)
  \,
  \diff \tau
  \leq
  C_1^{m,n}
$,
where $C_1^{m,n}$ is a positive constant. Estimating the right-hand side of~\eqref{eq:16} by its absolute value we also see that $\mathcal{Y}^{m,n}$ satisfies the inequality
\begin{equation}
  \label{eq:19}
  \dd{\mathcal{Y}^{m,n}}{t}
  \leq
  K^{m,n} \mathcal{Y}^{m,n}
  +
  \mathcal{H}^{m,n}
  .
\end{equation}
Inequality~\eqref{eq:16} and~Lemma~\ref{lm:1} then allow us to conclude that $\mathcal{Y}^{m,n} \xrightarrow{t \to + \infty} 0$, which is in virtue of the definition of $\mathcal{Y}^{m,n}$ tantamount to~\eqref{eq:1}. This is the desired stability result, see Theorem~\ref{thr:1} and also Corollary~\ref{crl:1} which gives the characterisation of the stability in terms of a Lebesgue space norm.

\section{Preliminaries}
\label{sec:preliminaries}
Let us now introduce the notation and let us give a precise formulation of the problem. The notation regarding the governing equations and their thermodynamic basis is the standard one, we mainly follow~\cite{bulcek.m.malek.j.ea:thermodynamics} and~\cite{malek.j.prusa.v:derivation}.

\enlargethispage{1.5em}

\subsection{Governing equations and basic facts from thermodynamics}
\label{sec:governing-equations}
The fluid of interest is assumed to be the standard incompressible Navier--Stokes--Fourier fluid. In particular, we do not use any form of Oberbeck--Boussinesq type approximation, we just consider a fluid with a constant density. The governing equations in the absence of external forces read
\begin{subequations}
  \label{eq:governing-equations-full}
  \begin{align}
    \label{eq:22}
    \divergence \vec{v} &= 0, \\
    \label{eq:23}
    \rho \dd{\vec{v}}{t} &= - \nabla p + \divergence\left(2 \mu \gradsym\right), \\
    \label{eq:24}
    \rho \cheatvolref \dd{\temp}{t} &= \divergence \left( \kapparef \nabla \temp \right) + 2 \mu \tensordot{\gradsym}{\gradsym},
  \end{align}
\end{subequations}
where $\temp$ denotes the absolute temperature, $\vec{v}$ denotes the velocity and $p$ stands for the pressure. Symbols $\rho$, $\mu$, $\cheatvolref$ and $\kapparef$ denote the density, the viscosity, the specific heat at constant volume and the thermal conductivity, and all these material parameters are assumed to be positive constants. Furthermore, the symbol
$  \gradsym =_{\bydefinition}
  \frac{1}{2}
  \left(
    \nabla \vec{v}
    +
    \transpose{\nabla \vec{v}}
  \right)
$
denotes the symmetric part of the velocity gradient, and $\dd{}{t} = \pd{}{t} + \vectordot{\vec{v}}{\nabla}$ denotes the material time derivative.

Since we need some basic thermodynamics in the construction of Lyapunov type functionals via the method proposed by~\cite{bulcek.m.malek.j.ea:thermodynamics}, we recall some basic facts concerning the thermodynamic basis of the governing equations. The specific Helmholtz free energy for the incompressible Navier--Stokes--Fourier fluid is given by the formula
  \begin{equation}
    \label{eq:free-energy-incompressible-nse}
    \fenergy 
    =_{\bydefinition} 
    - 
    \cheatvolref \temp \left(\ln \left( \frac{\temp}{\tempref} \right) - 1 \right).
  \end{equation}
Using the standard thermodynamical identities, see for example~\cite{callen.hb:thermodynamics}, the specific Helmholtz free energy $\fenergy$ in the form~\eqref{eq:free-energy-incompressible-nse} leads to the following formulae for the specific entropy $\entropy$ and the specific internal energy $\ienergy$,
\begin{equation}
  \label{eq:43}
  \entropy = \cheatvolref \ln \left( \frac{\temp}{\tempref} \right), \qquad
  \ienergy = \cheatvolref \temp.
\end{equation}
The heat flux vector $\hfluxc$ is given by the Fourier law,
\begin{equation}
 \label{eq:46}
 \hfluxc =_{\bydefinition} - \kapparef \nabla \temp,  
\end{equation}
and the Cauchy stress tensor is given as
$
\cstress 
=_{\bydefinition} 
-
p \identity
+ 
2 \mu \gradsym
$.
The general evolution equation for the entropy reads
\begin{equation}
  \label{eq:48}
  \rho
  \dd{\entropy}{t}
  =
  \frac{\entprodctemp_{\mathrm{th}}}{\temp}
  +
  \frac{\entprodctemp_{\mathrm{mech}}}{\temp}
  -
  \divergence
  \entfluxc
  ,
\end{equation}
where $\frac{\entprodctemp_{\mathrm{th}}}{\temp}$ and $\frac{\entprodctemp_{\mathrm{mech}}}{\temp}$ denote the entropy production terms related to the thermal and mechanical quantities, and $\entfluxc$ denotes the entropy flux. In our setting we have 
\begin{equation}
  \label{eq:49}
  \entprodctemp_{\mathrm{mech}}
  =_{\bydefinition}
  2 \mu \tensordot{\gradsym}{\gradsym}
  ,
  \qquad
  \entprodctemp_{\mathrm{th}}
  =_{\bydefinition}
  \kapparef
  \frac{\absnorm{\nabla \temp}^2}{\temp} 
  ,
  \qquad
  \entfluxc
  =_{\bydefinition}
  \frac{\hfluxc}{\temp}
  .
\end{equation}

\subsection{Boundary conditions}
\label{sec:boundary-conditions}

The boundary conditions are the standard boundary conditions for thermal convection. We consider a fluid in a rigid vessel $\Omega$ such that
\begin{equation}
  \label{eq:boundary-conditions}
  \left. \vec{v} \right|_{\partial \Omega} = \vec{0},
  \qquad
  \left. \temp \right|_{\partial \Omega} = \tempbdr,
\end{equation}
where $\tempbdr$ is a given function of position. In particular, \emph{we are interested in the case when $\tempbdr$ is a nontrivial function of position}. Since the temperature evolution equation~\eqref{eq:24} is a parabolic equation with a convective term and a positive source, we know that the temperature field~$\temp = \widehat{\temp} + \widetilde{\temp}$ is under given boundary conditions~\eqref{eq:boundary-conditions} \emph{bounded from below} uniformly in space and time, see for example~\cite{friedman.a:partial}, \cite{ladyzhenskaya.oa.solonnikov.va.ea:linear} or~\cite{lieberman.gm:second}. In particular, if the boundary temperature is positive $\tempbdr$, then $\temp$ is also always positive.

The no-slip boundary condition $\left. \vec{v} \right|_{\partial \Omega} = \vec{0}$ guarantees that the vessel is~\emph{mechanically isolated}. This means that the \emph{mechanical energy exchange} between the vessel and the outside environment, which is given by the surface integral $\int_{\partial \Omega} \vectordot{\cstress \vec{v}}{\vec{n}} \csurfacees$, is equal to zero,
\begin{equation}
  \label{eq:18}
  \int_{\partial \Omega} \vectordot{\cstress \vec{v}}{\vec{n}} \csurfacees = 0.
\end{equation}
On the other hand the \emph{heat exchange} between the vessel and the outside environment is given by the surface integral
$
-
\int_{\partial \Omega}
\vectordot{\hfluxc}{\vec{n}} \csurfacees
=
\int_{\partial \Omega}
\vectordot{\kappa_{\reference} \nabla \temp}{\vec{n}} \csurfacees
$.
The temperature boundary condition $\left. \temp \right|_{\partial \Omega} = \tempbdr$ does not in general imply that the \emph{temperature gradient} vanishes on the boundary as well, hence we in general expect that
\begin{equation}
  \label{eq:20}
  \int_{\partial \Omega}
  \vectordot{\hfluxc}{\vec{n}} \csurfacees
  \not
  =
  0
  .
\end{equation}
In this sense we characterise our system of interest as a mechanically isolated and thermally open system, and our choice of boundary conditions indeed corresponds to the system depicted in Figure~\ref{fig:systems-walls}.

\subsection{Spatially inhomogeneous non-equilibrium steady state}
\label{sec:non-equil-steady}

The non-equilibrium steady state $\widehat{\vec{W}} =_{\bydefinition} [\widehat{\vec{v}}, \widehat{\temp}, \widehat{p}]$ solves the system
\begin{equation}
  \label{eq:governing-equations-steady-state}
    0  = -\nabla{\widehat{p}}, \qquad
    0  = \divergence \left( \kapparef \nabla \widehat{\temp} \right),
\end{equation}
with boundary conditions
\begin{equation}
  \label{eq:boundary-conditions-steady-state}
  \left. \widehat{\vec{v}} \right|_{\partial \Omega} = \vec{0},
  \qquad
  \left. \widehat{\temp} \right|_{\partial \Omega} = \tempbdr,
\end{equation}
which means that the fluid is at rest, $\widehat{\vec{v}} = \vec{0}$, and that the temperature field is given by the steady heat equation~\eqref{eq:governing-equations-steady-state} with Dirichlet boundary condition~\eqref{eq:boundary-conditions-steady-state}. If $\tempbdr$ is a nontrivial function of position, then $\widehat{\temp}$ is a spatially inhomogeneous bounded function.

\subsection{Evolution equations for a perturbation to the steady state and the stability problem}
\label{sec:evol-equat-pert}
The velocity/temperature/pressure fields $\vec{W} = [\vec{v}, \temp, p]$ are split into the steady part $\widehat{\vec{W}} = [\widehat{\vec{v}}, \widehat{\temp}, \widehat{p}]$ and the perturbation $\widetilde{\vec{W}} =_{\bydefinition} [\widetilde{\vec{v}}, \widetilde{\temp}, \widetilde{p}]$, that is
$
[\vec{v}, \temp, \widehat{p}]
  =
  [\widehat{\vec{v}}, \widehat{\temp}, \widehat{p}]
  +
  [\widetilde{\vec{v}}, \widetilde{\temp}, \widetilde{p}]
$.
The governing equations for the triple $\widetilde{\vec{W}} = [\widetilde{\vec{v}}, \widetilde{\temp}, \widetilde{p}]$ are
\begin{subequations}
  \label{eq:governing-equations-perturbation}
  \begin{align}
    \label{eq:33}
    \divergence \widetilde{\vec{v}} &= 0,
    \\
    \label{eq:34}
    \rho
    \left(
    \pd{
    \widetilde{\vec{v}}
    }
    {
    t
    }
    +
    \left(
    \vectordot{\widetilde{\vec{v}}}{\nabla}
    \right)
    \vec{\widetilde{\vec{v}}}
    \right)
    &=
      -
      \nabla \widetilde{p}
      +
      \divergence
      \left(
      2 \mu \widetilde{\gradsym}
      \right)
      ,
    \\
    \label{eq:35}
    \rho
    \cheatvolref
    \left(
    \pd{\widetilde{\temp}}{t}
    +
    \left(
    \vectordot{\widetilde{\vec{v}}}{\nabla}
    \right)
    \widetilde{\temp}
    \right)
    +
    \rho
    \cheatvolref
    \left(
    \vectordot{\widetilde{\vec{v}}}{\nabla}
    \right)
    \widehat{\temp}
                                    &=
                                      \divergence \left( \kapparef \nabla \widetilde{\temp} \right)
                                      +
                                      2 \mu \tensordot{\widetilde{\gradsym}}{\widetilde{\gradsym}}.
  \end{align}
\end{subequations}
These equations are straightforward to obtain from the fact that $\vec{W}$ as well as $\widehat{\vec{W}}$ solve the governing equations~\eqref{eq:governing-equations-full}. The initial conditions for the full evolution equations read 
$
\left. \vec{v} \right|_{t=0} = \vec{v}_{\initial}
$,
$
  \left. \temp \right|_{t=0} = \temp_{\initial}
$,
which implies that the initial conditions for the perturbation are
\begin{equation}
  \label{eq:initial-conditions-perturbation}
  \left. \widetilde{\vec{v}} \right|_{t=0} = \vec{v}_{\initial}, 
  \qquad
    \left. \widetilde{\temp} \right|_{t=0} = \temp_{\initial} - \widehat{\temp}.
  \end{equation}
The boundary conditions for the perturbation are derived from \eqref{eq:boundary-conditions}, which upon using the fact that the steady state fulfills the boundary condition~\eqref{eq:boundary-conditions-steady-state}, leads to
\begin{equation}
  \label{eq:eq:boundary-conditions-perturbation}
  \left. \widetilde{\vec{v}} \right|_{\partial \Omega} = \vec{0},
  \qquad
  \left. \widetilde{\temp} \right|_{\partial \Omega} = 0.
\end{equation}
  
The task is to \emph{show that the perturbation vanishes as the time goes to infinity}, that is 
\begin{equation}
  \label{eq:42}
\widetilde{\vec{W}} \xrightarrow{t \to + \infty} \vec{0}.
\end{equation}
The exact meaning of the convergence can be specified at will in the sense that it can be specifically tailored for the problem. In our setting and for the temperature perturbation the meaning of~\eqref{eq:42} is~\eqref{eq:1}.

\section{A candidate for a Lyapunov type functional}
\label{sec:lyap-like-funct}
\cite{bulcek.m.malek.j.ea:thermodynamics} have proposed a thermodynamically based method for a systematic construction of a candidate for a Lyapunov type functional. We exploit this method, and we show that although the candidate for a Lyapunov type functional is not a genuine Lyapunov functional---its time derivative can not be shown to be negative---it is still useful in the stability analysis. The key findings of this section are the following.

\begin{lemma}[Lyapunov type functional and its time derivative]
  \label{lm:2}
    Let us consider the velocity/temperature perturbation $\widetilde{\temp}$, $\widetilde{\vec{v}}$ governed by equations~\eqref{eq:governing-equations-perturbation}. Let us define the functional
  \begin{equation}
    \label{eq:54}
    \mathcal{V}_{\mathrm{meq}}
    \left(
      \left.
        \widetilde{\vec{W}}
      \right\|
      \widehat{\vec{W}}
    \right)
    =_{\bydefinition}
    \int_{\Omega}
    \left[
      \rho
      \cheatvolref
      \widehat{\temp}
      \left[
        \frac{\widetilde{\temp}}{\widehat{\temp}}
        -
        \ln \left( 1 + \frac{\widetilde{\temp}}{\widehat{\temp}} \right)
      \right]
      +
      \frac{1}{2} \rho \absnorm{\widetilde{\vec{v}}}^2
    \right]
    \,
    \cvolumee
    ,
  \end{equation}
  where we use the notation introduced in Section~\ref{sec:preliminaries}. The functional remains non-negative for all possible temperature perturbations $\widetilde{\temp} \in (-\widehat{\temp}, + \infty)$ and velocity perturbations $\widetilde{\vec{v}}$, and it vanishes if and only if the perturbations $\widetilde{\temp}$ and $\widetilde{\vec{v}}$ vanish everywhere in the domain of interest. Furthermore, the time derivative of the functional is given by the formula
  \begin{multline}
    \label{eq:55}
    \dd{}{t}
    \mathcal{V}_{\mathrm{meq}}
    \left(
      \left.
        \widetilde{\vec{W}}
      \right\|
      \widehat{\vec{W}}
    \right)
    =
    -
    \int_{\Omega}
    \kapparef
    \widehat{\temp}
    \vectordot{\nabla \ln \left(1 + \frac{\widetilde{\temp}}{\widehat{\temp}} \right)}{\nabla \ln \left(1 + \frac{\widetilde{\temp}}{\widehat{\temp}} \right)}
    \,
    \cvolumee
    -
    \int_{\Omega}
    \frac{2 \mu \tensordot{\widetilde{\gradsym}}{\widetilde{\gradsym}}}{1 + \frac{\widetilde{\temp}}{\widehat{\temp}}}
    \,
    \cvolumee
    \\
    +
    \int_{\Omega}
    \rho
    \widehat{\temp}    
    \vectordot{
      \widetilde{\vec{v}}
    }
    {
      \nabla
      \left[
        \cheatvolref \ln \left(1 + \frac{\widetilde{\temp}}{\widehat{\temp}} \right)
      \right]
    }
    \,
    \cvolumee
    .
  \end{multline}
\end{lemma}

\begin{lemma}[On boundedness of some integrals]
  \label{lm:3}
  Let us consider the velocity/temperature perturbation $\widetilde{\temp}$, $\widetilde{\vec{v}}$ governed by equations~\eqref{eq:governing-equations-perturbation}. Then the dissipative heating term possesses a finite time-space integral,
  \begin{subequations}
    \label{eq:56}
    \begin{equation}
      \label{eq:57}
      \int_{\tau=0}^{+\infty}
      \left(
        \int_{\Omega}
        2 \mu \tensordot{\widetilde{\gradsym}}{\widetilde{\gradsym}}
        \,
        \cvolumee
      \right)
      \,
      \diff \tau
      <
      + \infty
      .
    \end{equation}
    Furthermore, the following integrals are also finite
    \begin{align}
      \label{eq:58}
    \int_{\tau=0}^{+\infty}
    \left(
    \int_{\Omega}
    \frac{2 \mu \tensordot{\widetilde{\gradsym}}{\widetilde{\gradsym}}}{1 + \frac{\widetilde{\temp}}{\widehat{\temp}}}
    \,
    \cvolumee
    \right)
    \,
    \diff \tau
    &<
    +
      \infty
      ,
      \\
      \label{eq:59}
      \int_{\tau=0}^{+\infty}
      \left(
  \int_{\Omega}
  \kapparef
  \widehat{\temp}
  \vectordot{\nabla \ln \left( 1 + \frac{\widetilde{\temp}}{\widehat{\temp}}\right)}{\nabla \ln \left( 1 + \frac{\widetilde{\temp}}{\widehat{\temp}}\right)}
  \,
  \cvolumee
  \right)
  \,
  \diff \tau
  &<
    + \infty
    ,
      \\
      \label{eq:60}
        \int_{\tau=0}^{+\infty}
      \absnorm{
      \left(
  \int_{\Omega}
  \rho
      \cheatvolref
      \left(
    \vectordot
    {
      \nabla \widehat{\temp}
    }
    {
      \widetilde{\vec{v}}
    }
      \right)
      \ln \left( 1 + \frac{\widetilde{\temp}}{\widehat{\temp}}\right)
  \,
  \cvolumee
  \right)
      }
      \,
      \diff \tau
  &<
      + \infty
      .
    \end{align}
  \end{subequations}
\end{lemma}
Note that Lemma~\ref{lm:3} shows that all the terms on the right-hand side of~\eqref{eq:55} have bounded integrals if integrated with respect to time from zero to infinity. (Regarding the last term on the right-hand side see identity~\eqref{eq:102} and~\eqref{eq:60}.) 

\subsection{Construction of a candidate for a Lyapunov type functional}
\label{sec:constr-lyap-funct}
A candidate for a Lyapunov type functional is, according to~\cite{bulcek.m.malek.j.ea:thermodynamics}, defined as
\begin{equation}
  \label{eq:61}
  \mathcal{V}_{\mathrm{meq}}
  \left(
    \left.
      \widetilde{\vec{W}}
    \right\|
    \widehat{\vec{W}}
  \right)
  =
  -
  \left\{
    \entropyrellp_{\widehat{\temp}}(\left. \widetilde{\vec{W}} \right\| \widehat{\vec{W}})
    -
    \energyrellp  (\left. \widetilde{\vec{W}} \right\| \widehat{\vec{W}})
  \right\}
  ,
\end{equation}
where
\begin{subequations}
  \label{eq:62}
  \begin{align}
    \label{eq:63}
    \entropyrellp_{\widehat{\temp}}(\left. \widetilde{\vec{W}} \right\| \widehat{\vec{W}})
    &=
      _{\bydefinition}
      \netentropy_{\widehat{\temp}} 
      \left(
      \widehat{\vec{W}}
      +
      \widetilde{\vec{W}}
      \right)
      -
      \netentropy_{\widehat{\temp}} 
      \left(
      \widehat{\vec{W}}
      \right)
      -
      \left.
      \Diff \netentropy_{\widehat{\temp}}
      \left(
      \vec{W}
      \right)
      \right|_{\vec{W}  = \widehat{\vec{W}}}
      \left[
      \widetilde{\vec{W}}
      \right]
      ,
    \\
    \label{eq:64}
    \energyrellp  (\left. \widetilde{\vec{W}} \right\| \widehat{\vec{W}})
    &=
      _{\bydefinition}
      \nettenergy 
      \left(
      \widehat{\vec{W}}
      +
      \widetilde{\vec{W}}
      \right)
      -
      \nettenergy 
      \left(
      \widehat{\vec{W}}
      \right)
      -
      \left.
      \Diff \nettenergy
      \left(
      \vec{W}
      \right)
      \right|_{\vec{W}  = \widehat{\vec{W}}}
      \left[
      \widetilde{\vec{W}}
      \right]
      ,
  \end{align}
\end{subequations}
and the functionals $\netentropy_{\widehat{\temp}} \left(\vec{W}\right)$ and $\nettenergy\left(\vec{W}\right)$ are defined as
\begin{equation}
  \label{eq:65}
    \netentropy_{\widehat{\temp}} \left(\vec{W}\right)
    =
      _{\bydefinition}
      \int_{\Omega}
      \rho
      \widehat{\temp}
      \entropy(\vec{W})
      \,
      \cvolumee
      ,
    \qquad
    \nettenergy\left(\vec{W}\right)
    =
      _{\bydefinition}
      \int_{\Omega}
      \left(
      \rho
      \ienergy(\vec{W})
      +
      \frac{1}{2} \rho \absnorm{\vec{v}}^2
      \right)
      \,
      \cvolumee
      .
\end{equation}
The symbols
$
\left.
\Diff \netentropy_{\widehat{\temp}}
\left(
  \vec{W}
\right)
\right|_{\vec{W}  = \widehat{\vec{W}}}
\left[
  \widetilde{\vec{W}}
\right]
$
and
$
\left.
\Diff \nettenergy
\left(
  \vec{W}
\right)
\right|_{\vec{W}  = \widehat{\vec{W}}}
\left[
  \widetilde{\vec{W}}
\right]
$
denote the G\^ateaux derivative of the given functionals at point $\widehat{\vec{W}}$ in the direction $\widetilde{\vec{W}}$. In our case, the specific entropy~$\entropy$ and specific internal energy~$\ienergy$ are given by formulae~\eqref{eq:43}, and the G\^ateaux derivatives of the functionals $\netentropy_{\widehat{\temp}} \left(\vec{W}\right)$ and $\nettenergy\left(\vec{W}\right)$ read
  \begin{equation}
    \label{eq:72}
    \left.
    \Diff[\vec{W}] \netentropy_{\widehat{\temp}} (\vec{W})
    \right|_{
    \vec{W} = \widehat{\vec{W}}
    }
    \left[
    \widetilde{\vec{W}}
    \right]
    =
      \int_{\Omega}
      \rho
      \cheatvolref \widetilde{\temp}
      \,
      \cvolumee
      ,
      \qquad
    \left.
    \Diff[\vec{W}] \nettenergy (\vec{W})
    \right|_{
    \vec{W} = \widehat{\vec{W}}
    }
    \left[
    \widetilde{\vec{W}}
    \right]
    =
      \int_{\Omega}
      \rho
      \cheatvolref \widetilde{\temp}
      \,
      \cvolumee
      .
    \end{equation}
    (Recall that we are investigating the stability of the non-equilibrium steady state $\widehat{\vec{W}}$ that does not evolve in time and where $\widehat{\vec{v}}=\vec{0}$.) Consequently, it is straightforward to see that the particular formulae for the functionals~$\entropyrellp_{\widehat{\temp}}(\left. \widetilde{\vec{W}} \right\| \widehat{\vec{W}})$ and $\energyrellp  (\left. \widetilde{\vec{W}} \right\| \widehat{\vec{W}})$ read
\begin{equation}
  \label{eq:74}
  \entropyrellp_{\widehat{\temp}}(\left. \widetilde{\vec{W}} \right\| \widehat{\vec{W}})
  =
  \int_{\Omega}
  \rho
  \cheatvolref
  \widehat{\temp}
  \left[
    \ln \left( 1 + \frac{\widetilde{\temp}}{\widehat{\temp}} \right)
    -
    \frac{\widetilde{\temp}}{\widehat{\temp}}
  \right]
  \,
  \cvolumee
  ,
  \quad
  \energyrellp  (\left. \widetilde{\vec{W}} \right\| \widehat{\vec{W}})
  =
  \int_{\Omega}
  \frac{1}{2} \rho \absnorm{\widetilde{\vec{v}}}^2
  \,
  \cvolumee
  .
\end{equation}
Using~\eqref{eq:74} in~\eqref{eq:61} we get the formula~\eqref{eq:54} for the candidate for a Lyapunov type functional.
It is straightforward to check that the proposed functional vanishes if and only if the perturbation vanishes, and that the functional is non-negative.

\subsection{Time derivative of our candidate for a Lyapunov type functional}
\label{sec:time-deriv-lyap}

It remains to investigate the time derivative of our candidate for a Lyapunov type functional~$\mathcal{V}_{\mathrm{meq}}$. If we were dealing with a genuine Lyapunov functional we would need to show that it is negative. Using the explicit formulae for the entropy and internal energy~\eqref{eq:43} it follows that
\begin{multline}
  \label{eq:77}
  \dd{}{t}
  \mathcal{V}_{\mathrm{meq}}
  \left(
    \left.
      \widetilde{\vec{W}}
    \right\|
    \widehat{\vec{W}}
  \right)
  =
  -
  \dd{}{t}
    \int_{\Omega}
    \rho
    \cheatvolref
    \widehat{\temp}
    \ln
    \left(
      \frac{\widehat{\temp} + \widetilde{\temp}}{\tempref}
    \right)
    \,
    \cvolumee
    +
    \dd{}{t}
    \int_{\Omega}
    \rho
    \cheatvolref
    \widehat{\temp}
    \ln
    \left(
      \frac{\widehat{\temp}}{\tempref}
    \right)
    \,
    \cvolumee
    \\
    +
    \dd{}{t}
    \int_{\Omega}
    \rho
    \cheatvolref
    \widetilde{\temp}
    \,
    \cvolumee
  +
  \dd{}{t}
  \nettenergy 
  \left(
    \widehat{\vec{W}}
    +
    \widetilde{\vec{W}}
  \right)
  -
  \dd{}{t}
  \nettenergy 
  \left(
    \widehat{\vec{W}}
  \right)
  -
  \dd{}{t}
  \int_{\Omega}
  \rho
  \cheatvolref
  \widetilde{\temp}
  \,
  \cvolumee
  -
  \int_{\Omega}
  \rho \vectordot{\widehat{\vec{v}}}{\widetilde{\vec{v}}}
  \,
  \cvolumee
  \\
  =
  -
  \int_{\Omega}
  \rho
  \widehat{\temp}
  \pd{\relentropy}{t}
  \,
  \cvolumee
  +
  \dd{}{t}
  \nettenergy 
  \left(
    \widehat{\vec{W}}
    +
    \widetilde{\vec{W}}
  \right)
  ,
\end{multline}
where we have introduced the relative entropy
\begin{equation}
  \label{eq:78}
  \relentropy =_{\bydefinition} \entropy(\widehat{\vec{W}} + \widetilde{\vec{W}}) - \entropy(\widehat{\vec{W}}).
\end{equation}
Note that in our case the explicit formula \eqref{eq:43} for the entropy implies that
\begin{equation}
  \label{eq:79}
  \relentropy = \cheatvolref \ln \left( 1 + \frac{\widetilde{\temp}}{\widehat{\temp}}\right).
\end{equation}
Consequently, we can work with a very simple formula
\begin{equation}
  \label{eq:80}
  \dd{}{t}
  \mathcal{V}_{\mathrm{meq}}
  \left(
    \left.
      \widetilde{\vec{W}}
    \right\|
    \widehat{\vec{W}}
  \right)
  =
  -
  \int_{\Omega}
  \rho
  \widehat{\temp}
  \pd{\relentropy}{t}
  \,
  \cvolumee
  +
  \dd{}{t}
  \nettenergy 
  \left(
    \widehat{\vec{W}}
    +
    \widetilde{\vec{W}}
  \right)
  .
\end{equation}

In order to get an explicit formula for the time derivative of the relative entropy $\relentropy$ in~\eqref{eq:80} we have to use the evolution equation for the entropy. The evolution equation for the entropy is~\eqref{eq:48}, which implies that\footnote{Recall that the material time derivative in~\eqref{eq:87} is taken with respect to the perturbed velocity field, that is 
  $    
  \dd{\entropy}{t}(\widehat{\vec{W}} + \widetilde{\vec{W}})
  =
  \pd{\entropy}{t}(\widehat{\vec{W}} + \widetilde{\vec{W}})
  +
  \vectordot{
    \left(\widehat{\vec{v}} + \widetilde{\vec{v}}\right)
  }
  {
    \nabla
    \entropy (\widehat{\vec{W}} + \widetilde{\vec{W}})
  }
  $, 
  while the material time derivative in~\eqref{eq:86} is taken with respect to the reference steady velocity field, that is
  $
  \dd{\entropy}{t}(\widehat{\vec{W}})
  =
  \pd{\entropy}{t}(\widehat{\vec{W}})
  +
  \vectordot{
    \widehat{\vec{v}}
  }
  {
    \nabla
    \entropy (\widehat{\vec{W}})
  }
  $%
  .} 
\begin{subequations}
  \label{eq:85}
  \begin{align}
    \label{eq:86}
    \rho
    \dd{\entropy}{t}(\widehat{\vec{W}})
    &=
      \frac{\entprodctemp_{\mathrm{th}} \left( \widehat{\vec{W}} \right)}{\widehat{\temp}}
      +
      \frac{\entprodctemp_{\mathrm{mech}} \left( \widehat{\vec{W}} \right)}{\widehat{\temp}}
      -
      \divergence
      \left(
      \frac{\hfluxc(\widehat{\vec{W}})}{\widehat{\temp}}
      \right)
      ,
    \\
    \label{eq:87}
    \rho
    \dd{\entropy}{t}(\widehat{\vec{W}} + \widetilde{\vec{W}})
    &=
      \frac{\entprodctemp_{\mathrm{th}} \left( \widehat{\vec{W}} + \widetilde{\vec{W}} \right)}{\widehat{\temp} + \widetilde{\temp}}
      +
      \frac{\entprodctemp_{\mathrm{mech}} \left( \widehat{\vec{W}} + \widetilde{\vec{W}} \right)}{\widehat{\temp} + \widetilde{\temp}}
      -
      \divergence
      \left(
      \frac{\hfluxc(\widehat{\vec{W}} + \widetilde{\vec{W}})}{\widehat{\temp} + \widetilde{\temp}}
      \right)
      .
  \end{align}
\end{subequations}
The relation between the temperature and the entropy is, for the model we are interested in, given by the formula~\eqref{eq:43}. Furthermore, the formula for the thermal part of the entropy production $\entprodctemp_{\mathrm{th}}$~\eqref{eq:49}, while the heat flux $\hfluxc$ is given by the Fourier law, which implies that
\begin{equation}
  \label{eq:88}
    \frac{\entprodctemp_{\mathrm{th}} (\vec{W})}{\temp}
    =
      \frac{\kapparef}{\cheatvolref^2}
      \vectordot{\nabla \entropy (\vec{W})}{\nabla \entropy (\vec{W})}
      ,
    \qquad
    \frac{\hfluxc(\vec{W})}{\temp}
    =
    - \frac{\kapparef}{\cheatvolref} \nabla \entropy(\vec{W}).
\end{equation}
Observations~\eqref{eq:88} and evolution equations~\eqref{eq:85} yield the evolution equation for the relative entropy $\relentropy$ in the form
\begin{multline}
  \label{eq:91}
  \rho \pd{\relentropy}{t}
  =
  \frac{\kapparef}{\cheatvolref^2}
  \vectordot{\nabla \relentropy}{\nabla \relentropy}
  +
  \frac{2 \kapparef}{\cheatvolref^2}
  \vectordot{\nabla \relentropy}{\nabla \widehat{\entropy}}
  +
  \divergence
  \left(
    \frac{\kapparef}{\cheatvolref}
    \nabla \relentropy
  \right)
  \\
  +
  \left[
    \frac{\entprodctemp_{\mathrm{mech}} \left( \widehat{\vec{W}} + \widetilde{\vec{W}} \right)}{\widehat{\temp} + \widetilde{\temp}}
    -
    \frac{\entprodctemp_{\mathrm{mech}} \left( \widehat{\vec{W}} \right)}{\widehat{\temp}}
  \right]
  -
  \rho
  \left[
    \vectordot{
      \left(\widehat{\vec{v}} + \widetilde{\vec{v}} \right)
    }
    {
      \nabla \relentropy
    }
    +
    \vectordot{
      \widetilde{\vec{v}}
    }
    {
      \nabla \widehat{\entropy}
    }
  \right]
  .
\end{multline}
If we take into account that in our case we have $\widehat{\vec{v}} = \vec{0}$,  and that the mechanical part of the entropy production $\entprodctemp_{\mathrm{mech}}$ is given by~\eqref{eq:49}, then~\eqref{eq:91} reduces to
\begin{multline}
  \label{eq:92}
  \rho \pd{\relentropy}{t}
  =
  \frac{\kapparef}{\cheatvolref^2}
  \vectordot{\nabla \relentropy}{\nabla \relentropy}
  +
  \frac{2 \kapparef}{\cheatvolref^2}
  \vectordot{\nabla \relentropy}{\nabla \widehat{\entropy}}
  +
  \divergence
  \left(
    \frac{\kapparef}{\cheatvolref}
    \nabla \relentropy
  \right)
  +
  \frac{2 \mu \tensordot{\widetilde{\gradsym}}{\widetilde{\gradsym}}}{\widehat{\temp} + \widetilde{\temp}}
  \\
  -
  \rho
  \left[
    \vectordot{
      \widetilde{\vec{v}}
    }
    {
      \nabla \relentropy
    }
    +
    \vectordot{
      \widetilde{\vec{v}}
    }
    {
      \nabla \widehat{\entropy}
    }
  \right]
  .
\end{multline}
On the other hand, the evolution equations for the net total energy $\nettenergy$ read
\begin{subequations}
  \label{eq:93}
  \begin{align}
    \label{eq:94}
    \dd{}{t}
    \nettenergy
    \left(
    \widehat{\vec{W}}
    +
    \widetilde{\vec{W}}
    \right)
    &=
      \int_{\Omega}
      \divergence
      \left(
      \cstress
      \left(
      \widehat{\vec{W}}
      +
      \widetilde{\vec{W}}
      \right)
      \left(
      \widehat{\vec{v}}
      +
      \widetilde{\vec{v}}
      \right)
      \right)
      \,
      \cvolumee
      -
      \int_{\Omega}
      \divergence
      \hfluxc
      \left(
      \widehat{\vec{W}}
      +
      \widetilde{\vec{W}}
      \right)
      \,
      \cvolumee
      ,
    \\
    \label{eq:95}
    \dd{}{t}
    \nettenergy
    \left(
    \widehat{\vec{W}}
    \right)
    &=
      \int_{\Omega}
      \divergence
      \left(
      \cstress
      \left(
      \widehat{\vec{W}}
      \right)
      \widehat{\vec{v}}
      \right)
      \,
      \cvolumee
      -
      \int_{\Omega}
      \divergence
      \hfluxc
      \left(
      \widehat{\vec{W}}
      \right)
      \,
      \cvolumee
      .
  \end{align}
\end{subequations}
Note that since we deal with a steady non-equilibrium state $\widehat{\vec{W}}$, then the left-hand side of~\eqref{eq:95} is in fact zero. Subtracting~\eqref{eq:94} and~\eqref{eq:95} yields
\begin{equation}
  \label{eq:96}
  \dd{}{t}
  \nettenergy
  \left(
    \widehat{\vec{W}}
    +
    \widetilde{\vec{W}}
  \right)
  =
  \int_{\Omega}
  \divergence \left( \kapparef \nabla \widetilde{\temp} \right)
  \,
  \cvolumee
  ,
\end{equation}
where we have used the specific constitutive relations for the incompressible Navier--Stokes--Fourier fluid, and the fact that~$\widehat{\vec{v}} = \vec{0}$ and that $\widetilde{\vec{v}}$ vanishes on the boundary. 
Now we are in the position to substitute into~\eqref{eq:80}. Using~\eqref{eq:96} and~\eqref{eq:92} in~\eqref{eq:80} yields
\begin{multline}
  \label{eq:97}
  \dd{}{t}
  \mathcal{V}_{\mathrm{meq}}
  \left(
    \left.
      \widetilde{\vec{W}}
    \right\|
    \widehat{\vec{W}}
  \right)
  =
  -
  \int_{\Omega}
  \frac{\kapparef}{\cheatvolref^2}
  \widehat{\temp}
  \vectordot{\nabla \relentropy}{\nabla \relentropy}
  \,
  \cvolumee
  -
  \int_{\Omega}
  \frac{2 \kapparef}{\cheatvolref^2}
  \widehat{\temp}
  \vectordot{\nabla \relentropy}{\nabla \widehat{\entropy}}
  \,
  \cvolumee
  \\
  -
  \int_{\Omega}
  \widehat{\temp}
  \divergence
  \left(
    \frac{\kapparef}{\cheatvolref}
    \nabla \relentropy
  \right)
  \,
  \cvolumee
  -
  \int_{\Omega}
  2 \mu \tensordot{\widetilde{\gradsym}}{\widetilde{\gradsym}} \frac{\widehat{\temp}}{\widehat{\temp} + \widetilde{\temp}}
  \,
  \cvolumee
  \\
  +
  \int_{\Omega}
  \rho
  \widehat{\temp}
  \left[
    \vectordot{
      \widetilde{\vec{v}}
    }
    {
      \nabla \relentropy
    }
    +
    \vectordot{
      \widetilde{\vec{v}}
    }
    {
      \nabla \widehat{\entropy}
    }
  \right]
  \,
  \cvolumee
  +
  \int_{\Omega}
  \divergence \left( \kapparef \nabla \widetilde{\temp} \right)
  \,
  \cvolumee
  .
\end{multline}

Now we can utilise the lesson learned in the analysis of the simpler heat conduction problem in a rigid body, see~\cite{bulcek.m.malek.j.ea:thermodynamics} for details, and we can rewrite~\eqref{eq:97} as
\begin{multline}
  \label{eq:101}
  \dd{}{t}
  \mathcal{V}_{\mathrm{meq}}
  \left(
    \left.
      \widetilde{\vec{W}}
    \right\|
    \widehat{\vec{W}}
  \right)
  =
  -
  \int_{\Omega}
  \frac{\kapparef}{\cheatvolref^2}
  \widehat{\temp}
  \vectordot{\nabla \relentropy}{\nabla \relentropy}
  \,
  \cvolumee
  -
  \int_{\Omega}
  2 \mu \tensordot{\widetilde{\gradsym}}{\widetilde{\gradsym}} \frac{\widehat{\temp}}{\widehat{\temp} + \widetilde{\temp}}
  \,
  \cvolumee
  \\
  +
  \int_{\Omega}
  \rho
  \widehat{\temp}
  \vectordot{
    \widetilde{\vec{v}}
  }
  {
    \nabla \relentropy
  }
  \,
  \cvolumee
  ,
\end{multline}
where we have also used that fact that 
$
\int_{\Omega}
  \rho
  \widehat{\temp}
  \vectordot{
    \widetilde{\vec{v}}
  }
  {
    \nabla \widehat{\entropy}
  }
  \,
  \cvolumee
  =
  0
$.
(This again follows from Stokes theorem and the fact that $\widehat{\temp} \nabla \widehat{\entropy}$ can be rewritten as $\nabla F(\widehat{\temp})$.) This manipulation finishes the proof of Lemma~\ref{lm:2}. An alternative derivation of the formula for the time derivative of the proposed functional is given in Appendix~\ref{sec:altern-deriv-form}, where we in fact derive the \emph{pointwise evolution equation} for the thermal part of the integrand in the functional $\mathcal{V}_{\mathrm{meq}}$.

Note that the last term in~\eqref{eq:101} can be also rewritten in a different form. Using the Stokes theorem and the fact that the perturbation velocity $\widetilde{\vecv}$ as well as $\Hdiff$ vanish on the boundary, and that the perturbation velocity field has zero divergence, we see that
\begin{equation}
  \label{eq:102}
  \int_{\Omega}
  \rho
  \widehat{\temp}
  \vectordot{
    \widetilde{\vec{v}}
  }
  {
    \nabla \relentropy
  }
  \,
  \cvolumee
  =
  -
  \int_{\Omega}
  \rho
  \left(
    \vectordot
    {
      \nabla \widehat{\temp}
    }
    {
      \widetilde{\vec{v}}
    }
  \right)
  \relentropy
  \,
  \cvolumee
  ,
\end{equation}
which among other things shows that the term vanishes for spatially homogeneous field $\widehat{\temp}$.

All the terms in~\eqref{eq:101} are at least ``quadratic'' in the perturbation as expected. The right-hand side of~\eqref{eq:101} \emph{does not contain boundary terms}, although we are working with a thermodynamically open system. This follows from the design of the proposed functional. We also see that the first two terms on the right-hand side of~\eqref{eq:101} have a sign, while the last term is indefinite. The presence of the last term prohibits one from showing that the time derivative of the functional $\mathcal{V}_{\mathrm{meq}}$ is, for a non-constant~$\widehat{\temp}$, a non-positive quantity. Consequently, $\mathcal{V}_{\mathrm{meq}}$ can not directly serve as a Lyapunov type functional.

\subsection{Estimate on the time derivative of the candidate for a Lyapunov type functional}
\label{sec:estim-time-deriv}
The evolution equation for the perturbation velocity is~\eqref{eq:34}, which upon testing by the perturbation velocity $\widetilde{\vecv}$ yields
\begin{equation}
  \label{eq:104}
  \frac{\rho}{2}
  \dd{}{t}
  \norm[\sleb{2}{\Omega}]{\widetilde{\vecv}}^2
  =
  -
  \int_{\Omega}
  2 \mu \tensordot{\widetilde{\gradsym}}{\widetilde{\gradsym}}
  \,
  \cvolumee.
\end{equation}
(Recall that $\widetilde{\vec{v}}$ satisfies zero Dirichlet boundary conditions.) Integrating~\eqref{eq:104} with respect to time, we get
\begin{equation}
  \label{eq:105}
  \frac{\rho}{2} \norm[\sleb{2}{\Omega}]{\widetilde{\vecv}}^2
  =
  \frac{\rho}{2} \left. \norm[\sleb{2}{\Omega}]{\widetilde{\vecv}}^2 \right|_{t=0}
  -
  \int_{\tau=0}^t
  \left(
    \int_{\Omega}
    2 \mu \tensordot{\widetilde{\gradsym}}{\widetilde{\gradsym}}
    \,
    \cvolumee
  \right)
  \,
  \diff \tau
  .
\end{equation}
Since $\frac{\rho}{2} \norm[\sleb{2}{\Omega}]{\widetilde{\vecv}}^2$ is a \emph{positive} quantity, we see that~\eqref{eq:105} implies that the dissipated energy is finite,
$
\int_{\tau=0}^t
  \left(
    \int_{\Omega}
    2 \mu \tensordot{\widetilde{\gradsym}}{\widetilde{\gradsym}}
    \,
    \cvolumee
  \right)
  \,
  \diff \tau
  <
  + \infty.
$
This is an important, yet trivial finding, and we will formulate it, for the sake of further reference, as Lemma~\ref{lm:3}, equation~\eqref{eq:57}.

Moreover, using the standard manipulation based on Poincar\'e and Korn (in)equality\footnote{Auxiliary tools from the theory of function spaces are for convenience summarised in Appendix~\ref{sec:auxiliary-tools}.}, we obtain the following estimate on the time derivative of the norm~$\norm[\sleb{2}{\Omega}]{\widetilde{\vecv}}^2$ 
\begin{equation}
  \label{eq:107}
  \frac{\rho}{2} \dd{}{t} \norm[\sleb{2}{\Omega}]{\widetilde{\vecv}}^2 \leq - \frac{\mu}{C_P} \norm[\sleb{2}{\Omega}]{\widetilde{\vecv}}^2,
\end{equation}
which yields
\begin{equation}
  \label{eq:108}
  \norm[\sleb{2}{\Omega}]{\widetilde{\vecv}}^2 \leq  \left. \norm[\sleb{2}{\Omega}]{\widetilde{\vecv}}^2 \right|_{t=0} \exponential{- \frac{2 \mu}{C_P \rho}t},
\end{equation}
which is the standard result, see for example~\cite{serrin.j:on}. \emph{Note that if we were dealing for example with a temperature dependent viscosity, it would be sufficient to have the viscosity that is bounded from below, and we would still get an exponential decay of the net kinetic energy.}

Using~\eqref{eq:108} and the $\varepsilon$--Young inequality $ab \leq \varepsilon a^2 + \frac{b^2}{4\varepsilon}$ then gives us
\begin{equation}
  \label{eq:109}
  \absnorm{
  \int_{\Omega}
  \rho
  \left(
  \vectordot{
    \widetilde{\vec{v}}
  }
  {
    \nabla 
    \widehat{\temp}
  }
\right)
  \relentropy
  \,
  \cvolumee
  }
  \leq
  \varepsilon
  \rho
  \max_{\vec{x} \in \Omega}
  \absnorm{\nabla \widehat{\temp}}
  \norm[\sleb{2}{\Omega}]{\relentropy}^2
  +
  \frac{\rho}{4\varepsilon}
  \max_{\vec{x} \in \Omega}
  \absnorm{\nabla \widehat{\temp}}
  \left. \norm[\sleb{2}{\Omega}]{\widetilde{\vecv}}^2 \right|_{t=0} \exponential{- \frac{2 \mu}{C_P \rho}t}
  .
\end{equation}
Furthermore, using again Poincar\'e inequality we get
\begin{equation}
  \label{eq:110}
  -
  \int_{\Omega}
  \frac{\kapparef}{\cheatvolref^2}
  \widehat{\temp}
  \vectordot{\nabla \relentropy}{\nabla \relentropy}
  \,
  \cvolumee
  \leq
  -
  \left(
    \min_{\vec{x} \in \Omega}
    \widehat{\temp}
  \right)
  \frac{\kapparef}{C_P \cheatvolref^2}
  \norm[\sleb{2}{\Omega}]{\Hdiff}^2
  ,
\end{equation}
and utilising~\eqref{eq:110} and~\eqref{eq:109} in the formula~\eqref{eq:101} for the time derivative of the proposed Lyapunov type functional we arrive at the inequality
\begin{subequations}
  \label{eq:111}
\begin{multline}
  \label{eq:112}
  \dd{}{t}
  \mathcal{V}_{\mathrm{meq}}
  \left(
    \left.
      \widetilde{\vec{W}}
    \right\|
    \widehat{\vec{W}}
  \right)
  \leq
  -
  \left(
    \frac{\kapparef}{C_P \cheatvolref^2}
    \min_{\vec{x} \in \Omega}
    \widehat{\temp}
    -
    \varepsilon
    \rho
    \max_{\vec{x} \in \Omega}
    \absnorm{\nabla \widehat{\temp}}
  \right)
  \norm[\sleb{2}{\Omega}]{\Hdiff}^2
  \\
  -
  \int_{\Omega}
  2 \mu \tensordot{\widetilde{\gradsym}}{\widetilde{\gradsym}} \frac{\widehat{\temp}}{\widehat{\temp} + \widetilde{\temp}}
  \,
  \cvolumee
  +
  \frac{\rho}{4 \varepsilon}
  \max_{\vec{x} \in \Omega}
  \absnorm{\nabla \widehat{\temp}}
  \left. \norm[\sleb{2}{\Omega}]{\widetilde{\vecv}}^2 \right|_{t=0} \exponential{- \frac{2 \mu}{C_P \rho}t}
  ,
\end{multline}
where the coefficient multiplying the term $\norm[\sleb{2}{\Omega}]{\Hdiff}^2$ that is
\begin{equation}
  \label{eq:113}
  \coeffC
  =_{\bydefinition}
    \frac{\kapparef}{C_P \cheatvolref^2}
    \min_{\vec{x} \in \Omega}
    \widehat{\temp}
    -
    \varepsilon
    \rho
    \max_{\vec{x} \in \Omega}
    \absnorm{\nabla \widehat{\temp}}
    ,
  \end{equation}
\end{subequations}
can be made positive by a suitable choice of the parameter $\varepsilon$. (Note that~\eqref{eq:113} is always positive if the steady state temperature field $\widehat{\temp}$ is spatially homogeneous.) Therefore, the structure of the estimate on the time derivative is a favourable one. It contains non-positive terms and a positive term that decays in time to zero.

\subsection{Boundedness of some integrals}
\label{sec:bound-some-integr}
Inequality~\eqref{eq:112} can be gainfully exploited in the proof of the boundedness of some integrals, see Lemma~\ref{lm:3}. The last term on the right-hand side of~\eqref{eq:112} is integrable with respect to time, which follows from the direct integration of the exponential. The integration of~\eqref{eq:112} with respect to time yields
\begin{multline}
  \label{eq:115}
  \mathcal{V}_{\mathrm{meq}}
  \left(
    \left.
      \widetilde{\vec{W}}
    \right\|
    \widehat{\vec{W}}
  \right)
  \leq
  \left.
  \mathcal{V}_{\mathrm{meq}}
  \left(
    \left.
      \widetilde{\vec{W}}
    \right\|
    \widehat{\vec{W}}
  \right)
  \right|_{t=0}
  -
  \int_{\tau=0}^{t}
  \coeffC
  \norm[\sleb{2}{\Omega}]{\Hdiff}^2
  \,
  \diff \tau
  \\
  -
  \int_{\tau=0}^{t}
  \left(
    \int_{\Omega}
    \frac{2 \mu \tensordot{\widetilde{\gradsym}}{\widetilde{\gradsym}}}{1 + \frac{\widetilde{\temp}}{\widehat{\temp}}}
    \,
    \cvolumee
  \right)
  \,
  \diff \tau
  +
  \int_{\tau=0}^{t}
  \frac{\rho}{4 \varepsilon}
  \max_{\vec{x} \in \Omega}
  \absnorm{\nabla \widehat{\temp}}
  \left. \norm[\sleb{2}{\Omega}]{\widetilde{\vecv}}^2 \right|_{t=0} \exponential{- \frac{2 \mu}{C_P \rho}\tau}
  \,
  \diff
  \tau
  ,
\end{multline}
and since $\mathcal{V}_{\mathrm{meq}}$ is a non-negative quantity, we see that the first two integrals on the right-hand side must be bounded. In particular, we get
\begin{equation}
  \label{eq:116}
  \int_{\tau=0}^{+\infty}
  \left(
    \int_{\Omega}
    \frac{2 \mu \tensordot{\widetilde{\gradsym}}{\widetilde{\gradsym}}}{1 + \frac{\widetilde{\temp}}{\widehat{\temp}}}
    \,
    \cvolumee
  \right)
  \,
  \diff \tau
  <
  +
  \infty
  ,
  \qquad
  \int_{\tau=0}^{+\infty}
  \norm[\sleb{2}{\Omega}]{\Hdiff}^2
  \,
  \diff \tau
  <
  +
  \infty
  .
\end{equation}
Now we can use~\eqref{eq:116} in~\eqref{eq:109}, and we can conclude that
\begin{equation}
  \label{eq:120}
  \int_{\tau=0}^{+\infty}
  \absnorm{
  \int_{\Omega}
  \rho
  \left(
    \vectordot
    {
      \widetilde{\vec{v}}      
    }
    {
      \nabla \widehat{\temp}
    }
  \right)
  \relentropy
  \,
  \cvolumee
  }
  \,
  \diff \tau
  <
  + \infty
  ,
\end{equation}
which in virtue of~\eqref{eq:102} also means that
$
\int_{\tau=0}^{+\infty}
  \absnorm{
    \int_{\Omega}
    \rho
    \widehat{\temp}
    \vectordot{
      \widetilde{\vec{v}}
    }
    {
      \nabla \relentropy
    }
    \,
    \cvolumee
  }
  \,
  \diff \tau
  <
  + \infty
$.
Finally, once we have obtained~\eqref{eq:120} we can go back to~\eqref{eq:101}, and use the same argument as before, and show~\eqref{eq:58} and~\eqref{eq:59}. Indeed, integrating~\eqref{eq:101} with respect to time we get
\begin{multline}
  \label{eq:123}
  \mathcal{V}_{\mathrm{meq}}
  \left(
    \left.
      \widetilde{\vec{W}}
    \right\|
    \widehat{\vec{W}}
  \right)
  =
  \left.
  \mathcal{V}_{\mathrm{meq}}
  \left(
    \left.
      \widetilde{\vec{W}}
    \right\|
    \widehat{\vec{W}}
  \right)
  \right|_{t=0}
  -
  \int_{\tau=0}^{+\infty}
  \left(
  \int_{\Omega}
  \frac{\kapparef}{\cheatvolref^2}
  \widehat{\temp}
  \vectordot{\nabla \relentropy}{\nabla \relentropy}
  \,
  \cvolumee
\right)
\,
\diff \tau
\\
-
  \int_{\tau=0}^{+\infty}
  \left(
  \int_{\Omega}
  \frac{2 \mu \tensordot{\widetilde{\gradsym}}{\widetilde{\gradsym}}}{1 + \frac{\widetilde{\temp}}{\widehat{\temp}}}
  \,
  \cvolumee
\right)
\,
\diff \tau
  +
  \int_{\tau=0}^{+\infty}
  \left(
  \int_{\Omega}
  \rho
  \widehat{\temp}
  \vectordot{
    \widetilde{\vec{v}}
  }
  {
    \nabla \relentropy
  }
  \,
  \cvolumee
\right)
\,
\diff \tau
,
\end{multline}
and since $\mathcal{V}_{\mathrm{meq}}$ is a non-negative quantity, and the last term on the right-hand side of~\eqref{eq:123} is bounded, we see that the first two integrals on the right-hand side must be bounded. The manipulations described above give the proof of Lemma~\ref{lm:3}.

\subsection{Remarks}
\label{sec:remarks-1}

Although we have not succeeded with the concept of Lyapunov functional we are close to the desired stability result. The time derivative of the proposed functional is negative up to an exponentially decaying positive term, see~\eqref{eq:112}. Moreover, the formula~\eqref{eq:112} which reads 
\begin{multline}
  \label{eq:124}
  \dd{}{t}
  \int_{\Omega}
  \left[
    \rho
    \cheatvolref
    \widehat{\temp}
    \left[
      \frac{\widetilde{\temp}}{\widehat{\temp}}
      -
      \ln \left( 1 + \frac{\widetilde{\temp}}{\widehat{\temp}} \right)
    \right]
    +
    \frac{1}{2} \rho \absnorm{\widetilde{\vec{v}}}^2
  \right]
  \,
  \cvolumee
  \leq
  \frac{\rho}{4 \varepsilon}
  \max_{\vec{x} \in \Omega}
  \widehat{\temp}
  \left. \norm[\sleb{2}{\Omega}]{\widetilde{\vecv}}^2 \right|_{t=0} \exponential{- \frac{2 \mu}{C_P \rho}t}
  \\  
  -
  \coeffC
  \int_{\Omega}
  \left(
    \cheatvolref
    \ln
    \left(
      1
      +
      \frac{\widetilde{\temp}}{\widehat{\temp}}
    \right)
  \right)^2
  \,
  \cvolumee
  -
  \int_{\Omega}
  2 \mu \tensordot{\widetilde{\gradsym}}{\widetilde{\gradsym}} \frac{\widehat{\temp}}{\widehat{\temp} + \widetilde{\temp}}
  \,
  \cvolumee
  .
\end{multline}
is also almost the right one from the perspective of Lemma~\ref{lm:1}. Indeed, if $x \in (-1, x_{\mathrm{crit}})$ then
\begin{equation}
  \label{eq:125}
  -
  \left[
    \ln
    \left(
      1
      +
      x
    \right)
  \right]^2
  \leq
  -
  \left[
    x
    -
    \ln \left( 1 + x \right)
  \right]
  ,
\end{equation}
where $x_{\mathrm{crit}}$ is a positive number. (See Figure~\ref{fig:lyapunov-sketch-a} for a sketch of the functions.)
\begin{figure}[t]
  \centering
  \subfloat[\label{fig:lyapunov-sketch-a}Graphs of functions
    $
    -
    \left(
      \ln
      \left(
        1
        +
        x
      \right)
    \right)^2$
    and
    $
    -
    \left(
      x
      -
      \ln \left( 1 + x \right)
    \right)
    $%
    .
    The graphs intersect at the point $x_{\mathrm{crit}} \approx 5.00914$.
    ]{
    \includegraphics[width=0.46\textwidth]{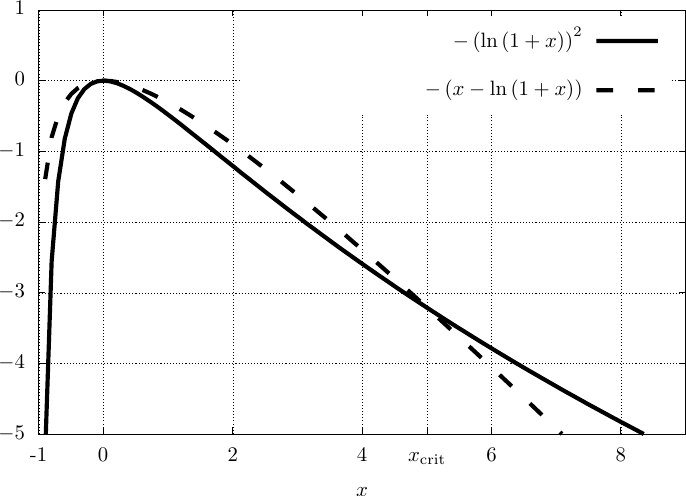}
  }
  \qquad
  \subfloat[\label{fig:lyapunov-sketch-b}
  Graphs of functions
  $
  -
  (
    (
      (
        1
        +
        x
      )^{\frac{m}{2}}
      -
      1
      )^2
    +
    (
      (
        1
        +
        x
      )^{\frac{n}{2}}
      -
      1
    )^2
  )
  $
  and
  $
  -
  n
  (
    \frac{1}{n}
    (
      1 + x
    )^{n}
    -
    \frac{1}{m}
    (
      1 + x
    )^{m}
    +
    \frac{n-m}{mn}
  )
  $. Functions are plotted for $n = \frac{8}{10}$ and $m = \frac{7}{10}$. 
  ]{\includegraphics[width=0.46\textwidth]{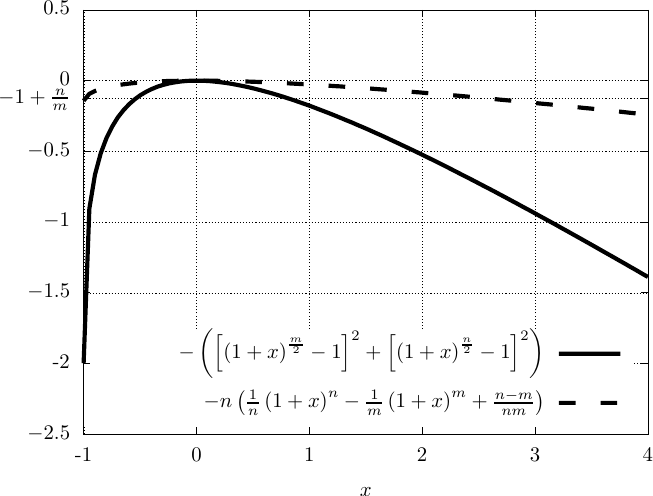}}
  \label{fig:lyapunov-sketch}
  \caption{Plots of auxiliary functions.}
\end{figure}
Consequently if the temperature ratio $\frac{\widetilde{\temp}}{\widehat{\temp}}$ is in the pointwise sense in the right interval, then we can rewrite~\eqref{eq:124} as
\begin{multline}
  \label{eq:126}
  \dd{}{t}
  \int_{\Omega}
  \left[
    \rho
    \cheatvolref
    \widehat{\temp}
    \left[
      \frac{\widetilde{\temp}}{\widehat{\temp}}
      -
      \ln \left( 1 + \frac{\widetilde{\temp}}{\widehat{\temp}} \right)
    \right]
    +
    \frac{1}{2} \rho \absnorm{\widetilde{\vec{v}}}^2
  \right]
  \,
  \cvolumee
  \leq
    \frac{\rho}{4 \varepsilon}
  \max_{\vec{x} \in \Omega}
  \widehat{\temp}
  \left. \norm[\sleb{2}{\Omega}]{\widetilde{\vecv}}^2 \right|_{t=0} \exponential{- \frac{2 \mu}{C_P \rho}t}
  \\
  -
  \coeffC
  \int_{\Omega}
  \cheatvolref^2
  \left[
    \frac{\widetilde{\temp}}{\widehat{\temp}}
    -
    \ln \left( 1 + \frac{\widetilde{\temp}}{\widehat{\temp}} \right)
  \right]
  \,
  \cvolumee
  -
  \int_{\Omega}
  2 \mu \tensordot{\widetilde{\gradsym}}{\widetilde{\gradsym}} \frac{\widehat{\temp}}{\widehat{\temp} + \widetilde{\temp}}
  \,
  \cvolumee
,
\end{multline}
which would allow us to directly exploit Lemma~\ref{lm:1}. (We would obtain the same integrand on the both sides of the inequality.) Unfortunately, we are not able to show that the temperature perturbation remains in the pointwise sense in the interval that guarantees the validity of~\eqref{eq:125}. The temperature perturbation $\widetilde{\temp}$ is only known to be uniformly bounded from \emph{below}, a uniform bound from above is not available.

\section{Rethinking the development of the candidate for a Lyapunov type functional}
\label{sec:reth-devel-lyap}
The inequality~\eqref{eq:125} that would guarantee the validity of the manipulations described above in Section~\ref{sec:remarks-1} seems to introduce a technical rather than a physical restriction. We will try to overcome this apparently unphysical restriction via a suitable choice of the temperature scale. The key findings of this section are the following.

\begin{lemma}[Family of Lyapunov type functionals and their time derivatives]
  \label{lm:4}
  Let us consider the velocity/temperature perturbation $\widetilde{\temp}$, $\widetilde{\vec{v}}$ governed by equations~\eqref{eq:governing-equations-perturbation}. Let us define the functionals
  \begin{equation}
    \label{eq:136}
    \mathcal{V}_{\mathrm{meq}}^{\vartemp,\, m}
    \left(
      \left.
        \widetilde{\vec{W}}
      \right\|
      \widehat{\vec{W}}
    \right)
    =_{\bydefinition}
    \int_{\Omega}
    \rho
    \cheatvolref
    \widehat{\temp}
    \left[
      \frac{\widetilde{\temp}}{\widehat{\temp}}
      -
      \frac{1}{m}
      \left(
        \left(
          1 + \frac{\widetilde{\temp}}{\widehat{\temp}}
        \right)^{m}
        -
        1
      \right)
    \right]
    \,
    \cvolumee
    +
    \int_{\Omega}
    \frac{1}{2} \rho \absnorm{\widetilde{\vec{v}}}^2
    \,
    \cvolumee   
    ,
  \end{equation}
  where we use the notation introduced in Section~\ref{sec:preliminaries}, and where $m \in (0,1)$. For fixed $m \in (0,1)$ any functional in family~\eqref{eq:136} remains non-negative for all possible temperature perturbations $\widetilde{\temp} \in (-\widehat{\temp}, + \infty)$, and it vanishes if and only if the perturbation~$\widetilde{\temp}$ vanishes everywhere in the domain of interest. Furthermore, the time derivative of the functionals is given by the formula
  \begin{multline}
    \label{eq:137}
    \dd{}{t}
    \mathcal{V}_{\mathrm{meq}}^{\vartemp,\, m}
    \left(
      \left.
        \widetilde{\vec{W}}
      \right\|
      \widehat{\vec{W}}
    \right)
    =
    -
    \int_{\Omega}
    4
    \frac{1-m}{m^2}
    \kapparef
    \widehat{\temp}
    \vectordot{
      \nabla
      \left[
        \left(
          1
          +
          \frac{\widetilde{\temp}}{\widehat{\temp}}
        \right)^{\frac{m}{2}}
        -
        1
      \right]
    }
    {
      \nabla
      \left[
        \left(
          1
          +
          \frac{\widetilde{\temp}}{\widehat{\temp}}
        \right)^{\frac{m}{2}}
        -
        1
      \right]
    }
    \,
    \cvolumee
    \\
    -
    \int_{\Omega}
    \frac{
      2 \mu \tensordot{\widetilde{\gradsym}}{\widetilde{\gradsym}}
    }
    {
      \left(
        1
        +
        \frac{\widetilde{\temp}}{\widehat{\temp}}
      \right)^{1-m}
    }
    \,
    \cvolumee
    -
    \int_{\Omega}
    \rho
    \cheatvolref
    \left(
      \vectordot{\nabla \widehat{\temp}}{\widetilde{\vec{v}}}
    \right)
    \frac{1-m}{m}
    \left[
      \left(
        1
        +
        \frac{\widetilde{\temp}}{\widehat{\temp}}
      \right)^m
      -
      1
    \right]
    .
  \end{multline}
\end{lemma}

\begin{lemma}[On boundedness of some integrals]
  \label{lm:5}
  Let us consider the velocity/temperature perturbation $\widetilde{\temp}$, $\widetilde{\vec{v}}$ governed by equations~\eqref{eq:governing-equations-perturbation}, and let $m \in (0,1)$, then
\begin{subequations}
  \begin{align}
      \label{eq:138}
    \int_{\tau=0}^{+\infty}
    \left(
        \int_{\Omega}
    \frac{
      2 \mu \tensordot{\widetilde{\gradsym}}{\widetilde{\gradsym}}
    }
    {
      \left(
        1
        +
        \frac{\widetilde{\temp}}{\widehat{\temp}}
      \right)^{1-m}
    }
    \,
    \cvolumee
    \right)
    \,
    \diff \tau
    &<
    +
      \infty
      ,
      \\
      \label{eq:139}
      \int_{\tau=0}^{+\infty}
    \left(
        \int_{\Omega}
    4
    \frac{1-m}{m^2}
    \kapparef
    \widehat{\temp}
    \vectordot{
      \nabla
      \left[
        \left(
          1
          +
          \frac{\widetilde{\temp}}{\widehat{\temp}}
        \right)^{\frac{m}{2}}
        -
        1
      \right]
    }
    {
      \nabla
      \left[
        \left(
          1
          +
          \frac{\widetilde{\temp}}{\widehat{\temp}}
        \right)^{\frac{m}{2}}
        -
        1
      \right]
    }
    \,
    \cvolumee
  \right)
  \,
  \diff \tau
  &<
    + \infty
    ,
      \\
      \label{eq:140}
        \int_{\tau=0}^{+\infty}
    \absnorm{
        \int_{\Omega}
    \rho
    \cheatvolref
    \left(
      \vectordot{\nabla \widehat{\temp}}{\widetilde{\vec{v}}}
    \right)
    \frac{1-m}{m}
    \left[
      \left(
        1
        +
        \frac{\widetilde{\temp}}{\widehat{\temp}}
      \right)^m
      -
      1
    \right]
  \,
  \cvolumee
  }
  \,
  \diff \tau
  &<
      + \infty
      .
    \end{align}
  \end{subequations}
\end{lemma}

\subsection{Description based on an alternative temperature scale}
\label{sec:descr-based-modif}

So far we have used the absolute temperature scale $\temp$, and we have found that the corresponding natural Lyapunov type functional is almost the right one. Now we can speculate about the following idea. The fact that the perturbed temperature field returns back to the non-equilibrium steady state can not depend on the choice of the temperature scale. The temperature scale is nothing that is given~\emph{a priori}, the absolute temperature scale is used because of its convenience. (See for example~\cite{pippard.ab:elements} for a thorough discussion of the concept of temperature, and the theoretical benefits of the absolute temperature scale. See also~\cite{fosdick.rl.rajagopal.kr:on} for a discussion of fundamentals of thermometry.) In fact any strictly increasing function of the absolute temperature can be used as an alternative temperature scale. This alternative temperature scale will provide us essentially the same tool for \emph{qualitative comparison of hotness} of two bodies in the sense it will preserve ordering with respect to the hotness. On the other hand the \emph{quantitative aspects of the ordering} will differ, which might help us to overcome some technical difficulties in dealing with the governing equations. This idea motivates the following manipulation.

The non-equilibrium steady temperature field is given by the solution of steady heat equation $\divergence \left( \kapparef \nabla \temp \right) = 0$. If we rewrite the equation as
\begin{equation}
  \label{eq:161}
  \divergence \left( \kapparef \tempref \left(\frac{\temp}{\tempref}\right)^m \frac{1}{\left(\frac{\temp}{\tempref}\right)^m} \nabla \frac{\temp}{\tempref} \right) = 0,
\end{equation}
then nothing changes. Now we can introduce a new temperature scale $\vartemp$ as
\begin{equation}
  \label{eq:162}
  \frac{\vartemp}{\vartempref} =_{\bydefinition} \left( \frac{\temp}{\tempref} \right)^{1-m},
\end{equation}
where $m < 1$ is an arbitrary number. (The restriction to $m<1$ means that the new temperature scale preserves the ordering according to hotness. Values $m>1$ would lead to the inversion of ordering, since the new temperature would measure coldness instead of hotness, and the heat flux would change its sign.) Using this newly defined temperature scale, we see that~\eqref{eq:161} reads
\begin{equation}
  \label{eq:163}
  \divergence
  \left(
    \kapparef \frac{1}{1-m} \frac{\tempref}{\vartempref} \left( \frac{\vartemp}{\vartempref} \right)^{\frac{m}{1-m}} \nabla \vartemp 
  \right)
  =
  0
  .
\end{equation}
Consequently, if we decide to use the alternative temperature scale~\eqref{eq:162}, then the fluid of interest effectively behaves as a fluid with the \emph{temperature dependent thermal conductivity}
\begin{equation}
  \label{eq:164}
  \kappa(\vartemp) = \kapparef \frac{1}{1-m} \frac{\tempref}{\vartempref} \left( \frac{\vartemp}{\vartempref} \right)^{\frac{m}{1-m}},
\end{equation}
and the equation for the non-equilibrium steady state reads
$
\divergence \left( \kappa(\vartemp) \nabla \vartemp  \right) = 0.
$
If we make the same substitution in the complete evolution equation for the temperature~\eqref{eq:24}, then we see that~\eqref{eq:24} reads
\begin{equation}
  \label{eq:166}
  \rho \cheatvolref \frac{1}{1-m} \frac{\tempref}{\vartempref} \left( \frac{\vartemp}{\vartempref} \right)^{\frac{m}{1-m}} \dd{\vartemp}{t} = \divergence \left( \kappa(\vartemp) \nabla \vartemp \right) + 2 \mu \tensordot{\gradsym}{\gradsym}.
\end{equation}
Consequently, the fluid of interest effectively behaves as a fluid with the \emph{temperature dependent specific heat at constant volume}
\begin{equation}
  \label{eq:167}
  \cheatvol(\vartemp)
  =
  \cheatvolref \frac{1}{1-m} \frac{\tempref}{\vartempref} \left( \frac{\vartemp}{\vartempref} \right)^{\frac{m}{1-m}},
\end{equation}
and the evolution equation for the newly defined temperature reads
\begin{equation}
  \label{eq:168}
  \rho \cheatvol(\vartemp) \dd{\vartemp}{t} = \divergence \left( \kappa(\vartemp) \nabla \vartemp \right) + 2 \mu \tensordot{\gradsym}{\gradsym}
  .
\end{equation}

\emph{This means that if we use the alternative temperature scale then we are effectively dealing with a fluid with the temperature dependent heat conductivity and the temperature dependent specific heat capacity at constant volume.} Concerning this system of governing equations, we know how to construct a candidate for the Lyapunov functional. All that needs to be done is to identify the specific Helmholtz free energy, and use the method proposed by~\cite{bulcek.m.malek.j.ea:thermodynamics}. The specific Helmholtz free energy is a solution to the equation
\begin{equation}
  \label{eq:173}
\cheatvol(\vartemp) = -\vartemp \ppd{\fenergy}{\vartemp},
\end{equation}
which yields
$
\fenergy
  =
  -
  \cheatvolref
  \vartempref
  \frac{\tempref}{\vartempref}
  \frac{1}{m}
  \left[
    (1-m)
    \left(
      \frac{\vartemp}{\vartempref}
    \right)^{\frac{1}{1-m}}
    -
    \frac{\vartemp}{\vartempref}
  \right]
$  .
(Note that this formula gives us an additional restriction $m>0$. One of the integration constants in~\eqref{eq:173} has been chosen in such a way that the entropy vanishes at $\vartemp = \vartempref$. The other integration constant is immaterial since we are finally interested only in the differences between the corresponding quantities.) The other quantities such as the entropy and the internal energy are then found using standard thermodynamical identities. This yields
\begin{equation}
  \entropy
  =
    \cheatvolref
    \frac{\tempref}{\vartempref}
    \frac{1}{m}
    \left[
    \left(
    \frac{\vartemp}{\vartempref}
    \right)^{\frac{m}{1-m}}
    -
    1
    \right]
    ,
  \qquad
  \ienergy
  =
    \cheatvolref
    \vartempref
    \frac{\tempref}{\vartempref}
    \left(
    \frac{\vartemp}{\vartempref}
    \right)^{\frac{1}{1-m}}
    .
\end{equation}

The formula for the candidate for a Lyapunov functional which is obtained by the same method as in Section~\ref{sec:lyap-like-funct} reads
\begin{multline}
  \label{eq:177}
  \mathcal{V}_{\mathrm{meq}}^{\vartemp,\, m}
  \left(
    \left.
      \widetilde{\vec{W}}
    \right\|
    \widehat{\vec{W}}
  \right)
  =
    \int_{\Omega}
  \frac{1}{2} \rho \absnorm{\widetilde{\vec{v}}}^2
  \,
  \cvolumee
  \\
  +
  \int_{\Omega}
  \rho
  \cheatvolref
  \frac{\tempref}{\vartempref} \left( \frac{\widehat{\vartemp}}{\vartempref} \right)^{\frac{1}{1-m}}
  \vartempref
  \left[
    \left(
      \left(
        1 + \frac{\widetilde{\vartemp}}{\widehat{\vartemp}}
      \right)^{\frac{1}{1-m}}
      -
      1
    \right)
    -
    \frac{1}{m}
    \left(
      \left(
        1 + \frac{\widetilde{\vartemp}}{\widehat{\vartemp}}
      \right)^{\frac{m}{1-m}}
      -
      1
    \right)
  \right]
  \,
  \cvolumee
  .
\end{multline}
In virtue of~\eqref{eq:162} we also see that
$
1
    +
    \frac{\widetilde{\vartemp}}{\widehat{\vartemp}}
    =
      \left(
      1
      +
      \frac{\widetilde{\temp}}{\widehat{\temp}}
      \right)^{1-m}
$,
hence the functional in~\eqref{eq:177} can be rewritten in terms of the original absolute temperature~$\temp$ as
\begin{equation}
  \label{eq:190}
  \mathcal{V}_{\mathrm{meq}}^{\vartemp,\, m}
  \left(
    \left.
      \widetilde{\vec{W}}
    \right\|
    \widehat{\vec{W}}
  \right)
  =
  \int_{\Omega}
  \rho
  \cheatvolref
  \widehat{\temp}
  \left[
    \frac{\widetilde{\temp}}{\widehat{\temp}}
    -
    \frac{1}{m}
    \left(
      \left(
        1 + \frac{\widetilde{\temp}}{\widehat{\temp}}
      \right)^{m}
      -
      1
    \right)
  \right]
  \,
  \cvolumee
  +
  \int_{\Omega}
  \frac{1}{2} \rho \absnorm{\widetilde{\vec{v}}}^2
  \,
  \cvolumee
  .
\end{equation}
Furthermore, the formula for the time derivative $\mathcal{V}_{\mathrm{meq}}^{\vartemp,\, m}$ rewritten again in the terms of the original absolute temperature $\temp$ is the formula~\eqref{eq:137} reported in Lemma~\ref{lm:4}. The proof of formula~\eqref{eq:137} is given in Appendix~\ref{sec:altern-deriv-form}.

\subsection{Estimate of the time derivative of the prospective Lyapunov type functionals and boundedness of some integrals}
\label{sec:estim-time-deriv-1}

Using the whole family of functionals $\mathcal{V}_{\mathrm{meq}}^{\vartemp,\, m}$ we need to revisit the estimates established in Section~\ref{sec:lyap-like-funct}. The objective is to get the same type of estimates as in Section~\ref{sec:estim-time-deriv} and Section~\ref{sec:bound-some-integr}. Namely we will be first dealing with the last term on the right-hand side of~\eqref{eq:137}, and we will try to show that it has a bounded time integral. Let us denote
\begin{equation}
  \label{eq:191}
    d_{\mathrm{th},\, m}
  =_{\bydefinition}
  \left(
    1
    +
    \frac{\widetilde{\temp}}{\widehat{\temp}}
  \right)
  ^{\frac{m}{2}}
  -
  1
  .
\end{equation}
First, we note that in virtue of the boundary condition~\eqref{eq:eq:boundary-conditions-perturbation}, the function $d_{\mathrm{th}, \mathrm{m}}$ vanishes on the boundary, $\left. d_{\mathrm{th},\, m} \right|_{\partial \Omega} = 0$. Furthermore, we see that
$
\left(
    1
    +
    \frac{\widetilde{\temp}}{\widehat{\temp}}
  \right)^m
  -
  1
  =
    d_{\mathrm{th},\, m}^2
  +
  2  d_{\mathrm{th},\, m}
$.
Using the just introduced notation, we can rewrite~\eqref{eq:137} as
\begin{multline}
  \label{eq:193}
  \dd{}{t}
  \mathcal{V}_{\mathrm{meq}}^{\vartemp,\, m}
  \left(
    \left.
      \widetilde{\vec{W}}
    \right\|
    \widehat{\vec{W}}
  \right)
  =
  -
  4
  \frac{1-m}{m^2}
  \kapparef
  \int_{\Omega}
  \widehat{\temp}
  \absnorm
  {
    \nabla
      d_{\mathrm{th},\, m}
  }^2
  \,
  \cvolumee
  -
  \int_{\Omega}
  \frac{
    2 \mu \tensordot{\widetilde{\gradsym}}{\widetilde{\gradsym}}
  }
  {
    \left(
      1
      +
      \frac{\widetilde{\temp}}{\widehat{\temp}}
    \right)^{1-m}
  }
  \,
  \cvolumee
  \\
  -
  \frac{1-m}{m}
  \int_{\Omega}
  \rho
  \cheatvolref
  \left(
    \vectordot
    {\widetilde{\vec{v}}}
    {
      \nabla
      \widehat{\temp}    
    }
  \right)
  \left(
      d_{\mathrm{th},\, m}^2
    +
    2  d_{\mathrm{th},\, m}
  \right)
  \,
  \cvolumee
  .
\end{multline}
Let us see whether the last term without the sign can be ``absorbed'' in the first term. We see that the $\varepsilon$-Young inequality implies
\begin{equation}
  \label{eq:194}
  \absnorm{
    \int_{\Omega}
    \left(
      \vectordot
      {
        \widetilde{\vec{v}}
      }
      {
        \nabla \widehat{\temp}
      }
    \right)
      d_{\mathrm{th},\, m}
    \,
    \cvolumee
  }
  \leq
  \varepsilon
  \max_{\vec{x} \in \Omega}
  \absnorm{\nabla \widehat{\temp}}
  \norm[\sleb{2}{\Omega}]{  d_{\mathrm{th},\, m}}^2
  +
  \frac{1}{4\varepsilon}
  \max_{\vec{x} \in \Omega}
  \absnorm{\nabla \widehat{\temp}}
  \left. \norm[\sleb{2}{\Omega}]{\widetilde{\vecv}}^2 \right|_{t=0} \exponential{- \frac{2 \mu}{C_P \rho}t}
  ,
\end{equation}
where we have used the decay of the velocity perturbation~\eqref{eq:108}. Furthermore, we see that
\begin{equation}
  \label{eq:195}
  \absnorm{
    \int_{\Omega}
    \left(
      \vectordot
      {
        \widetilde{\vec{v}}
      }
      {
        \nabla \widehat{\temp}
      }
    \right)
      d_{\mathrm{th},\, m}^2
    \,
    \cvolumee
  }
  \leq
  \max_{\vec{x} \in \Omega} \absnorm{\nabla \widehat{\temp}}
  \int_{\Omega}
    d_{\mathrm{th},\, m}^2
  \absnorm{
    \widetilde{\vec{v}}
  }
  \,
  \cvolumee
  ,
\end{equation}
and we can focus on the term
$
\int_{\Omega}
  d_{\mathrm{th},\, m}^2
\absnorm{
  \widetilde{\vec{v}}
}
\,
\cvolumee
$
only.

Before we proceed with the estimates, we derive a simple inequality that will be useful in a moment. From the Sobolev embedding, see Lemma~\ref{lm:sobolev-embedding-continuous}, and the Poincar\'e inequality, see Lemma~\ref{lm:poincare}, we know that
$
\norm[\sleb{2s}{\Omega}]{d_{\mathrm{th},\, m}}
\leq
C_{\mathrm{S}} \norm[\ssob{1}{2}{\Omega}]{d_{\mathrm{th},\, m}}
\leq
C_{\mathrm{S}}
\sqrt{
  1
  +
  C_{\mathrm{P}}
}
\norm[\sleb{2}{\Omega}]{\nabla d_{\mathrm{th},\, m}}
$,
where $2s \in [1, p^\star]$. (Recall that in the three dimensional case we have $p^{\star}= 6$.) In particular, we see that for all $s \in \left[1, \frac{p^\star}{2} \right]$ we get
\begin{equation}
  \label{eq:197}
  \left(
    \int_{\Omega}
      d_{\mathrm{th},\, m}^{2s}
    \,
    \cvolumee
  \right)^{\frac{1}{s}}
  \leq
  C_{\mathrm{S}}^2
  \left(
    1
    +
    C_{\mathrm{P}}
  \right)
  \int_{\Omega}
  \absnorm{
    \nabla   d_{\mathrm{th},\, m}
  }
  ^2
  \,
  \cvolumee
  .
\end{equation}

Going back to the right-hand side of~\eqref{eq:195}, and using the H\"older inequality with $p_1=3$ and $p_2 = \frac{3}{2}$ we see that
\begin{equation}
  \label{eq:199}
  \int_{\Omega}
    d_{\mathrm{th},\, m}^2
  \absnorm{
    \widetilde{\vec{v}}
  }
  \,
  \cvolumee
  \leq
  \left(
    \int_{\Omega}
      d_{\mathrm{th},\, m}^{6}
    \,
    \cvolumee
  \right)^{\frac{1}{3}}
  \left(
    \int_{\Omega}
    \absnorm{
      \widetilde{\vec{v}}
    }^{\frac{3}{2}}
    \,
    \cvolumee
  \right)^{\frac{2}{3}}
  .
\end{equation}
The first integral on the right-hand side of~\eqref{eq:199} can be estimated using~\eqref{eq:197} where we choose $s=3$. The second integral can be estimated using the trivial embedding of Lebesgue spaces, see Lemma~\ref{lm:trivial-embedding}, in particular for $p_1 = \frac{3}{2}$ and $p_2 = 2$ we get
$
\norm[\sleb{\frac{3}{2}}{\Omega}]{\widetilde{\vec{v}}}
\leq
\absnorm{\Omega}^{\frac{1}{6}}
\norm[\sleb{2}{\Omega}]{\widetilde{\vec{v}}}
$.
Finally, we see that
\begin{equation}
  \label{eq:201}
  \absnorm{
    \int_{\Omega}
    \left(
      \vectordot
      {
        \widetilde{\vec{v}}
      }
      {
        \nabla \widehat{\temp}
      }
    \right)
      d_{\mathrm{th},\, m}^2
    \,
    \cvolumee
  }
  \leq
  =
  \max_{\vec{x} \in \Omega} \absnorm{\nabla \widehat{\temp}}
  C_{\mathrm{S}}^2
  \left(
    1
    +
    C_{\mathrm{P}}
  \right)
  \absnorm{\Omega}^{\frac{1}{6}}
  \left. \norm[\sleb{2}{\Omega}]{\widetilde{\vecv}} \right|_{t=0} \exponential{- \frac{\mu}{C_P \rho}t}
  \norm[\sleb{2}{\Omega}]{\nabla   d_{\mathrm{th},\, m}}^2.
\end{equation}
where we have again used the decay of the velocity perturbation~\eqref{eq:108}.

Using the derived estimates~\eqref{eq:201} and~\eqref{eq:194} we revisit the expression for the time derivative of the proposed Lyapunov type functional~\eqref{eq:193}, and we see that the right-hand side of~\eqref{eq:193} can be estimated as follows
\begin{multline}
  \label{eq:202}
  \dd{}{t}
  \mathcal{V}_{\mathrm{meq}}^{\vartemp,\, m}
  \left(
    \left.
      \widetilde{\vec{W}}
    \right\|
    \widehat{\vec{W}}
  \right)
  \leq
  -
  2
  \frac{1-m}{m^2}
  \kapparef
  \min_{\vec{x} \in \Omega} \widehat{\temp}
  \int_{\Omega}
  \absnorm
  {
    \nabla
      d_{\mathrm{th},\, m}
  }^2
  \,
  \cvolumee
  \\
  -
  2
  \frac{1-m}{m^2}
  \kapparef
  \min_{\vec{x} \in \Omega} \widehat{\temp}
  \frac{1}{C_P}
  \int_{\Omega}
  \absnorm
  {
      d_{\mathrm{th},\, m}
  }^2
  \,
  \cvolumee
  -
  \int_{\Omega}
  \frac{
    2 \mu \tensordot{\widetilde{\gradsym}}{\widetilde{\gradsym}}
  }
  {
    \left(
      1
      +
      \frac{\widetilde{\temp}}{\widehat{\temp}}
    \right)^{1-m}
  }
  \,
  \cvolumee
  \\
  +
  2
  \frac{1-m}{m}
  \rho
  \cheatvolref
  \varepsilon
  \max_{\vec{x} \in \Omega}
  \absnorm{\nabla \widehat{\temp}}
  \norm[\sleb{2}{\Omega}]{  d_{\mathrm{th},\, m}}^2
  +
  \frac{1-m}{m}
  \rho
  \cheatvolref
  \frac{1}{2\varepsilon}
  \max_{\vec{x} \in \Omega}
  \absnorm{\nabla \widehat{\temp}}
  \left. \norm[\sleb{2}{\Omega}]{\widetilde{\vecv}}^2 \right|_{t=0} \exponential{- \frac{2 \mu}{C_P \rho}t}
  \\
  +
  \frac{1-m}{m}
  \rho
  \cheatvolref
  \max_{\vec{x} \in \Omega} \absnorm{\nabla \widehat{\temp}}
  C_{\mathrm{S}}^2
  \left(
    1
    +
    C_{\mathrm{P}}
  \right)
  \absnorm{\Omega}^{\frac{1}{6}}
  \left. \norm[\sleb{2}{\Omega}]{\widetilde{\vecv}} \right|_{t=0} \exponential{- \frac{\mu}{C_P \rho}t}
  \norm[\sleb{2}{\Omega}]{\nabla   d_{\mathrm{th},\, m}}^2
  ,
\end{multline}
where we have split the first gradient term on the right-hand side of~\eqref{eq:193} to halves, and where we have used Poincar\'e inequality. Inequality~\eqref{eq:202} can be further rewritten as
\begin{subequations}
\label{eq:203}  
\begin{multline}
  \label{eq:204}
  \dd{}{t}
  \mathcal{V}_{\mathrm{meq}}^{\vartemp,\, m}
  \left(
    \left.
      \widetilde{\vec{W}}
    \right\|
    \widehat{\vec{W}}
  \right)
  \leq
  -
  C_{\nabla   d_{\mathrm{th},\, m}}
  \norm[\sleb{2}{\Omega}]{\nabla   d_{\mathrm{th},\, m}}^2
  -
  C_{d_{\mathrm{th},\, m}}
  \norm[\sleb{2}{\Omega}]{  d_{\mathrm{th},\, m}}^2
  -
  \int_{\Omega}
  \frac{
    2 \mu \tensordot{\widetilde{\gradsym}}{\widetilde{\gradsym}}
  }
  {
    \left(
      1
      +
      \frac{\widetilde{\temp}}{\widehat{\temp}}
    \right)^{1-m}
  }
  \,
  \cvolumee
  \\
  +
    \frac{1-m}{m}
  \rho
  \cheatvolref
  \frac{1}{2\varepsilon}
  \max_{\vec{x} \in \Omega}
  \absnorm{\nabla \widehat{\temp}}
  \left. \norm[\sleb{2}{\Omega}]{\widetilde{\vecv}}^2 \right|_{t=0} \exponential{- \frac{2 \mu}{C_P \rho}t}
  ,
\end{multline}
where
\begin{align}
  \label{eq:205}
  C_{\nabla   d_{\mathrm{th},\, m}}
  &=_{\bydefinition}
    2
    \frac{1-m}{m^2}
    \kapparef
    \min_{\vec{x} \in \Omega} \widehat{\temp}
    -
    \frac{1-m}{m}
    \rho
    \cheatvolref
    \max_{\vec{x} \in \Omega} \absnorm{\nabla \widehat{\temp}}
    C_{\mathrm{S}}^2
    \left(
    1
    +
    C_{\mathrm{P}}
    \right)
    \absnorm{\Omega}^{\frac{1}{6}}
    \left. \norm[\sleb{2}{\Omega}]{\widetilde{\vecv}} \right|_{t=0} \exponential{- \frac{\mu}{C_P \rho}t}
    ,
  \\
  \label{eq:206}
  C_{  d_{\mathrm{th},\, m}}
  &=_{\bydefinition}
    2
    \frac{1-m}{m^2}
    \frac{\kapparef}{C_P}  
    \min_{\vec{x} \in \Omega} \widehat{\temp}
    -
    2
    \frac{1-m}{m}
    \rho
    \cheatvolref
    \varepsilon
    \max_{\vec{x} \in \Omega}
    \absnorm{\nabla \widehat{\temp}}
    .
\end{align}
\end{subequations}
First we note that both the coefficients $C_{\nabla   d_{\mathrm{th},\, m}}$ and $C_{  d_{\mathrm{th},\, m}}$ are positive in the case of spatially homogeneous temperature distribution~$\widehat{\temp}$, and that the last term on the right-hand side vanishes in this case as well. Consequently, $\mathcal{V}_{\mathrm{meq}}^{\vartemp,\, m}$ is a genuine Lyapunov type functional if we are interested in the steady state with spatially homogeneous temperature distribution~$\widehat{\temp}$.

Moreover, if $\varepsilon$ in~\eqref{eq:194} is chosen sufficiently small we see that $C_{  d_{\mathrm{th},\, m}}$ is positive even if the steady state temperature distribution $\widehat{\temp}$ is spatially inhomogeneous. Furthermore, the negative part of the coefficient $C_{\nabla   d_{\mathrm{th},\, m}}$ is a decreasing function of time, while its positive part remains constant. The coefficient $C_{\nabla   d_{\mathrm{th},\, m}}$ therefore becomes positive once we reach a (possibly large) time threshold, and then it remains positive.

Finally, the last term in~\eqref{eq:204} is integrable with respect to time from zero to infinity. Having derived~\eqref{eq:204} we are in fact in the same position as in Section~\ref{sec:lyap-like-funct}, see especially formula \eqref{eq:111}. In particular, we can follow the manipulations in Section~\ref{sec:bound-some-integr} and show that all the terms on the right-hand side of \eqref{eq:137} are finite, which proves Lemma~\ref{lm:5}.

\section{Application of the lemma on the decay of integrable functions}
\label{sec:appl-lemma-decay}

From Section~\ref{sec:lyap-like-funct} and Section~\ref{sec:reth-devel-lyap} we know that for $m \in [0,1)$ we have a family of functionals~$\mathcal{V}_{\mathrm{meq}}^{\vartemp,\, m}$ that have favourable properties. (The results obtained in Section~\ref{sec:lyap-like-funct} can be seen as a formal limit $m \to 0+$ of the results obtained in Section~\ref{sec:reth-devel-lyap}.) The functionals are non-negative for any perturbation, and they vanish if and only if the perturbation vanishes in the whole domain of interest. Furthermore, the key terms in the formulae for the time derivatives are negative. Unfortunately, none of these functionals alone are sufficient for establishing stability, namely none of them are suitable as a Lyapunov functional. However, \emph{we might still benefit from the fact that we have the whole family of functionals, and we can try to combine them in a suitable manner}.

In order to proceed with this idea further, we need two simple lemmas concerning the behaviour of certain functions, see Lemma~\ref{lm:6} and Lemma~\ref{lm:7} in Appendix~\ref{sec:auxiliary-tools-1}. Note that Lemma~\ref{lm:7} implies that for $m,n \in (0,1)$, $n>m> \frac{n}{2}$ and $x \in (-1, \infty)$ we have the inequality
\begin{equation}
  \label{eq:306}
    -
  \left(
    \left[
      \left(
        1
        +
        x
      \right)^{\frac{m}{2}}
      -
      1
    \right]^2
    +
    \left[
      \left(
        1
        +
        x
      \right)^{\frac{n}{2}}
      -
      1
    \right]^2
  \right)
  \leq
  -
  n
  \left[
    \frac{1}{n}
    \left(
      1 + x
    \right)^{n}
    -
    \frac{1}{m}
    \left(
      1 + x
    \right)^{m}
    +
    \frac{n-m}{mn}
  \right]
  .
\end{equation}
See also Figure~\ref{fig:lyapunov-sketch-b} for a sketch of the graphs of the corresponding functions.

Let us consider the family of functionals $\mathcal{V}_{\mathrm{meq}}^{\vartemp,\, m}$ introduced in Lemma~\ref{lm:4}, that is
\begin{subequations}
  \label{eq:217}
  \begin{equation}
    \label{eq:218}
    \mathcal{V}_{\mathrm{meq}}^{\vartemp,\, m}
  \left(
    \left.
      \widetilde{\vec{W}}
    \right\|
    \widehat{\vec{W}}
  \right)
    =
    \int_{\Omega}
    \rho
    \cheatvolref
    \widehat{\temp}
     v_{\mathrm{th},\, m}
    \,
    \cvolumee
    +
    \int_{\Omega}
    \frac{1}{2} \rho \absnorm{\widetilde{\vec{v}}}^2
    \,
    \cvolumee
    ,
  \end{equation}
  where we have introduced the symbol $v_{\mathrm{th},\, m}$, 
  \begin{equation}
    \label{eq:219}
    v_{\mathrm{th},\, m}
    =_{\bydefinition}
    \frac{\widetilde{\temp}}{\widehat{\temp}}
    -
    \frac{1}{m}
    \left(
      \left(
        1 + \frac{\widetilde{\temp}}{\widehat{\temp}}
      \right)^{m}
      -
      1
    \right)
  ,
\end{equation}
for the integrand in thermal part of~$\mathcal{V}_{\mathrm{meq}}^{\vartemp,\, m}$, and where $m \in (0,1)$.
\end{subequations}
We know, see Lemma~\ref{lm:4}, that the evolution equation for this functional reads
\begin{equation}
  \label{eq:220}
  \dd{}{t}
  \mathcal{V}_{\mathrm{meq}}^{\vartemp,\, m}
   \left(
    \left.
      \widetilde{\vec{W}}
    \right\|
    \widehat{\vec{W}}
  \right)
  =
  -
  \int_{\Omega}
  4
  \frac{1-m}{m^2}
  \kapparef
  \widehat{\temp}
  \vectordot{
    \nabla
    d_{\mathrm{th},\, m}
  }
  {
    \nabla
    d_{\mathrm{th},\, m}
  }
  \,
  \cvolumee
  -
  \mathcal{R}^{\vartheta, m}
  ,
\end{equation}
where we have introduced the notation
\begin{subequations}
  \label{eq:13}
  \begin{align}
    \label{eq:12}
    d_{\mathrm{th},\, m}
    &=_{\bydefinition}
      \left(
      1
      +
      \frac{\widetilde{\temp}}{\widehat{\temp}}
      \right)^{\frac{m}{2}}
      -
      1
      ,
    \\
    \label{eq:15}
    \mathcal{R}^{\vartheta, m}
    &=_{\bydefinition}
    \int_{\Omega}
    \frac{
    2 \mu \tensordot{\widetilde{\gradsym}}{\widetilde{\gradsym}}
      }
      {
      \left(
      1
      +
      \frac{\widetilde{\temp}}{\widehat{\temp}}
      \right)^{1-m}
      }
      \,
      \cvolumee
      +
      \int_{\Omega}
      \rho
      \cheatvolref
      \left(
      \vectordot{\nabla \widehat{\temp}}{\widetilde{\vec{v}}}
      \right)
      \frac{1-m}{m}
      \left[
      \left(
      1
      +
      \frac{\widetilde{\temp}}{\widehat{\temp}}
      \right)^m
      -
      1
      \right]
      \,
      \cvolumee
      ,
  \end{align}
\end{subequations}
already known from Section~\ref{sec:estim-time-deriv-1}. From Section~\ref{sec:reth-devel-lyap}, we also know that all the quantities on the right-hand side of~\eqref{eq:220} have finite time integrals, see Lemma~\ref{lm:3} and Lemma~\ref{lm:5}. (In particular all the integrals in the remainder term $\mathcal{R}^{\vartheta, m}$ have finite integral with respect to time.) Now we can subtract the evolution equations for $\mathcal{V}_{\mathrm{meq}}^{\vartemp,\, m}$ and  $\mathcal{V}_{\mathrm{meq}}^{\vartemp, n}$, which yields
\begin{multline}
  \label{eq:226}
  \dd{}{t}
  \left[
    \mathcal{V}_{\mathrm{meq}}^{\vartemp,\, m}
    \left(
      \left.
        \widetilde{\vec{W}}
      \right\|
      \widehat{\vec{W}}
    \right)
    -
    \mathcal{V}_{\mathrm{meq}}^{\vartemp, n}
    \left(
      \left.
        \widetilde{\vec{W}}
      \right\|
      \widehat{\vec{W}}
    \right)
  \right]
  =
  \dd{}{t}
  \int_{\Omega}
  \rho
  \cheatvolref
  \widehat{\temp}
  \left(
    v_{\mathrm{th},\, m}
    -
    v_{\mathrm{th},\, n}
  \right)
  \,
  \cvolumee
  \\
  =
  -
  \int_{\Omega}
  4
  \frac{1-m}{m^2}
  \kapparef
  \widehat{\temp}
  \absnorm{
    \nabla
    d_{\mathrm{th},\, m}
  }^2
  \,
  \cvolumee
  +
  \int_{\Omega}
  4
  \frac{1-n}{n^2}
  \kapparef
  \widehat{\temp}
  \absnorm{
    \nabla
    d_{\mathrm{th},\, n}
  }^2
  \,
  \cvolumee
  -
  \mathcal{R}^{\vartheta, m}
  +
  \mathcal{R}^{\vartheta, n}
  .
\end{multline}
Using the formulae for $v_{\mathrm{th},\, m}$ and $v_{\mathrm{th},\, n}$ we see that
\begin{equation}
  \label{eq:227}
  v_{\mathrm{th},\, m} - v_{\mathrm{th},\, n}
  =
  \frac{1}{n}
  \left(
    1 + \frac{\widetilde{\temp}}{\widehat{\temp}}
  \right)^{n}
  -
  \frac{1}{m}
  \left(
    1 + \frac{\widetilde{\temp}}{\widehat{\temp}}
  \right)^{m}
  +
  \frac{n-m}{mn}
  ,
\end{equation}
which upon using Lemma~\ref{lm:6} implies that if $n>m$ then
\begin{equation}
  \label{eq:228}
  v_{\mathrm{th},\, m} - v_{\mathrm{th},\, n} \geq 0, 
\end{equation}
while the equality holds if and only if $\frac{\widetilde{\temp}}{\widehat{\temp}}$ vanishes. (Let us recall that $\frac{\widetilde{\temp}}{\widehat{\temp}} \in (-1, + \infty)$, see Section~\ref{sec:boundary-conditions}.) Consequently, we see that the difference
\begin{equation}
  \label{eq:229}
  \mathcal{Y}^{m,n}
  \left(
    \left.
      \widetilde{\vec{W}}
    \right\|
    \widehat{\vec{W}}
  \right)
  =_{\bydefinition}
  \mathcal{V}_{\mathrm{meq}}^{\vartemp,\, m}
    \left(
      \left.
        \widetilde{\vec{W}}
      \right\|
      \widehat{\vec{W}}
    \right)
    -
    \mathcal{V}_{\mathrm{meq}}^{\vartemp, n}
    \left(
      \left.
        \widetilde{\vec{W}}
      \right\|
      \widehat{\vec{W}}
    \right)
  \end{equation}
  is for $n,m \in (0,1)$, $n>m$ a \emph{non-negative functional that vanishes if and only it the perturbation $\widetilde{\temp}$ vanishes}.

  Furthermore the right-hand side of~\eqref{eq:226}, which is the time derivative of $ \mathcal{Y}^{m,n}$, can be rewritten as
\begin{multline}
  \label{eq:230}
  \dd{}{t}
  \int_{\Omega}
  \rho
  \cheatvolref
  \widehat{\temp}
  \left(
    v_{\mathrm{th},\, m}
    -
    v_{\mathrm{th},\, n}
  \right)
  \,
  \cvolumee
  =
  -
  \int_{\Omega}
  4
  \frac{1-m}{m^2}
  \kapparef
  \widehat{\temp}
    \left(
    \absnorm{
      \nabla
      d_{\mathrm{th},\, m}
    }^2
    +
    \absnorm{
      \nabla
      d_{\mathrm{th},\, n}
    }^2
  \right)
  \,
  \cvolumee
  \\
  +
  \int_{\Omega}
  4
  \frac{1-m}{m^2}
  \kapparef
  \widehat{\temp}
  \absnorm{
    \nabla
    d_{\mathrm{th},\, n}
  }^2
  \,
  \cvolumee
  +
  \int_{\Omega}
  4
  \frac{1-n}{n^2}
  \kapparef
  \widehat{\temp}
  \absnorm{
    \nabla
    d_{\mathrm{th},\, n}
  }^2
  \,
  \cvolumee
  -
  \mathcal{R}^{\vartheta, m}
  +
  \mathcal{R}^{\vartheta, n}
  ,
\end{multline}
where we have added and subtracted the quantity
$
    \int_{\Omega}
  4
  \frac{1-m}{m^2}
  \kapparef
  \widehat{\temp}
  \absnorm{
    \nabla
    d_{\mathrm{th},\, n}
  }^2
  \,
  \cvolumee
$.
(We again recall that all the terms on the right-hand side have a finite integral if we integrate them from zero to infinity with respect to time.) Concerning the first term on the right-hand side of~\eqref{eq:230}, we can apply Poincar\'e inequality, and we get
\begin{equation}
  \label{eq:232}
  -
  \int_{\Omega}
  4
  \frac{1-m}{m^2}
  \kapparef
  \widehat{\temp}
  \left(
    \absnorm{
      \nabla
      d_{\mathrm{th},\, m}
    }^2
    +
    \absnorm{
      \nabla
      d_{\mathrm{th},\, n}
    }^2
  \right)
  \,
  \cvolumee
  \\
  \leq
  -
  4
  \frac{1-m}{m^2}
  \frac{
    \kapparef    
  }
  {
    C_P
  }
  \min_{\vec{x} \in \Omega} \widehat{\temp}
  \int_{\Omega}
  \left(
    d_{\mathrm{th},\, m}^2
    +
    d_{\mathrm{th},\, n}^2
  \right)
  \,
  \cvolumee
  .
\end{equation}
Now for  $n,m \in (0,1)$, $n>m> \frac{n}{2}$ we have the pointwise inequality
\begin{equation}
  \label{eq:233}
  -
  \left(
    d_{\mathrm{th},\, m}^2
    +
    d_{\mathrm{th},\, n}^2
  \right)
  \leq
  -
  n
  \left(
    v_{\mathrm{th},\, m}
    -
    v_{\mathrm{th},\, n}
  \right)
  .
\end{equation}
This is a straightforward consequence of Lemma~\ref{lm:7}, see~\eqref{eq:306}, and the definitions of $d_{\mathrm{th},\, m}$ and  $v_{\mathrm{th},\, m}$, see~\eqref{eq:12} and~\eqref{eq:219}. Once we have~\eqref{eq:233} we can go back to~\eqref{eq:232}, and we get the estimate
\begin{multline}
  \label{eq:236}
  -
  \int_{\Omega}
  4
  \frac{1-m}{m^2}
  \kapparef
  \widehat{\temp}
  \left(
    \absnorm{
      \nabla
      d_{\mathrm{th},\, m}
    }^2
    +
    \absnorm{
      \nabla
      d_{\mathrm{th},\, n}
    }^2
  \right)
  \,
  \cvolumee
  \\
  \leq
  -
  4
  n
  \frac{1-m}{m^2}
  \frac{\kapparef}{C_P}
  \frac{
    \min_{\vec{x} \in \Omega} \widehat{\temp}
  }
  {
    \max_{\vec{x} \in \Omega} \widehat{\temp}
  }
  \int_{\Omega}
  \widehat{\temp}
  \left(
    v_{\mathrm{th},\, m}
    -
    v_{\mathrm{th},\, n}
  \right)
  \,
  \cvolumee
  .
\end{multline}
Consequently, the inequality~\eqref{eq:230} reads
\begin{multline}
  \label{eq:237}
  \dd{}{t}
  \int_{\Omega}
  \rho
  \cheatvolref
  \widehat{\temp}
  \left(
    v_{\mathrm{th},\, m}
    -
    v_{\mathrm{th},\, n}
  \right)
  \,
  \cvolumee
  \leq
  -
  \frac{4 n \left(1-m \right)\kapparef}{m^2C_P}
  \frac{
    \min_{\vec{x} \in \Omega} \widehat{\temp}
  }
  {
    \max_{\vec{x} \in \Omega} \widehat{\temp}
  }
  \int_{\Omega}
  \widehat{\temp}
  \left(
    v_{\mathrm{th},\, m}
    -
    v_{\mathrm{th},\, n}
  \right)
  \,
  \cvolumee
  \\
  +
  \int_{\Omega}
  4
  \frac{1-m}{m^2}
  \kapparef
  \widehat{\temp}
  \absnorm{
    \nabla
    d_{\mathrm{th},\, n}
  }^2
  \,
  \cvolumee
  +
  \int_{\Omega}
  4
  \frac{1-n}{n^2}
  \kapparef
  \widehat{\temp}
  \absnorm{
    \nabla
    d_{\mathrm{th},\, n}
  }^2
  \,
  \cvolumee
  -
  \mathcal{R}^{\vartheta, m}
  +
  \mathcal{R}^{\vartheta, n}
  .
\end{multline}
This inequality implies that
\begin{subequations}
  \label{eq:238}
\begin{equation}
  \label{eq:239}
  \dd{}{t}
  \mathcal{Y}^{m,n}
  \left(
    \left.
      \widetilde{\vec{W}}
    \right\|
    \widehat{\vec{W}}
  \right)
  \leq
  -
  K^{m,n}
  \mathcal{Y}^{m,n}
  \left(
    \left.
      \widetilde{\vec{W}}
    \right\|
    \widehat{\vec{W}}
  \right)
  +
  \mathcal{H}^{m,n}
  \left(
    \left.
      \widetilde{\vec{W}}
    \right\|
    \widehat{\vec{W}}
  \right)
  ,
\end{equation}
where the positive constant $K^{m,n}$ is given by the formula 
$
K^{m,n}
  =_{\bydefinition}
  \frac{4 n \left(1-m \right)\kapparef}{m^2C_P}
  \frac{
    \min_{\vec{x} \in \Omega} \widehat{\temp}
  }
  {
    \max_{\vec{x} \in \Omega} \widehat{\temp}
  }
$,
and the time dependent \emph{non-negative} functional $\mathcal{H}^{m,n}$ is given by the formula
\begin{multline}
  \label{eq:241}
  \mathcal{H}^{m,n}
  \left(
    \left.
      \widetilde{\vec{W}}
    \right\|
    \widehat{\vec{W}}
  \right)
  =_{\bydefinition}
  \int_{\Omega}
  4
  \frac{1-m}{m^2}
  \kapparef
  \widehat{\temp}
  \absnorm{
    \nabla
    d_{\mathrm{th},\, n}
  }^2
  \,
  \cvolumee
  +
  \int_{\Omega}
  4
  \frac{1-n}{n^2}
  \kapparef
  \widehat{\temp}
  \absnorm{
    \nabla
    d_{\mathrm{th},\, n}
  }^2
  \,
  \cvolumee
  \\
  +
  \int_{\Omega}
  \frac{
      2 \mu \tensordot{\widetilde{\gradsym}}{\widetilde{\gradsym}}
    }
    {
      \left(
        1
        +
        \frac{\widetilde{\temp}}{\widehat{\temp}}
      \right)^{1-m}
    }
  \,
  \cvolumee
  +
  \int_{\Omega}
  \frac{
      2 \mu \tensordot{\widetilde{\gradsym}}{\widetilde{\gradsym}}
    }
    {
      \left(
        1
        +
        \frac{\widetilde{\temp}}{\widehat{\temp}}
      \right)^{1-n}
    }
  \,
  \cvolumee
  \\
  +
  \absnorm{
  \int_{\Omega}
  \rho
  \cheatvolref
  \left(
    \vectordot{\nabla \widehat{\temp}}{\widetilde{\vec{v}}}
  \right)
  \frac{1-m}{m}
  \left[
    \left(
      1
      +
      \frac{\widetilde{\temp}}{\widehat{\temp}}
    \right)^m
    -
    1
  \right]
  \,
  \cvolumee
}
\\
  +
  \absnorm{
  \int_{\Omega}
  \rho
  \cheatvolref
  \left(
    \vectordot{\nabla \widehat{\temp}}{\widetilde{\vec{v}}}
  \right)
  \frac{1-n}{n}
  \left[
    \left(
      1
      +
      \frac{\widetilde{\temp}}{\widehat{\temp}}
    \right)^n
    -
    1
  \right]
  \,
  \cvolumee
}
.
\end{multline}
\end{subequations}
We can observe that in virtue of Lemma~\ref{lm:5} we know that
$
\int_{\tau=0}^{+\infty}
  \mathcal{H}^{m,n}
  \left(
    \left.
      \widetilde{\vec{W}}
    \right\|
    \widehat{\vec{W}}
  \right)
  \,
  \diff \tau
  <
  +
  \infty
$.

Now we are finally in the position to use Lemma~\ref{lm:10} on the decay of integrable functions. First, we integrate~\eqref{eq:239} over the time interval~$(s,t)$, which yields
\begin{multline}
  \label{eq:243}
  \left.
    \mathcal{Y}^{m,n}
    \left(
      \left.
        \widetilde{\vec{W}}
      \right\|
      \widehat{\vec{W}}
    \right)
  \right|_{\tau=t}
  \leq
  \left.
    \mathcal{Y}^{m,n}
    \left(
      \left.
        \widetilde{\vec{W}}
      \right\|
      \widehat{\vec{W}}
    \right)
  \right|_{\tau=s}
  -
  K^{m,n}
  \int_{\tau=s}^{t}
  \mathcal{Y}^{m,n}
  \left(
    \left.
      \widetilde{\vec{W}}
    \right\|
    \widehat{\vec{W}}
  \right)
  \,
  \diff \tau
  \\
  +
  \int_{\tau=s}^{t}
  \mathcal{H}^{m,n}
  \left(
    \left.
      \widetilde{\vec{W}}
    \right\|
    \widehat{\vec{W}}
  \right)
  \,
  \diff \tau
  .
\end{multline}
If we take $s=0$ and $t=+\infty$, we see that the term on the left-hand side is nonegative, while the first and the last term on the right-hand side are finite and positive. Consequently, the negative term $\int_{\tau=0}^{+\infty}
  \mathcal{Y}^{m,n}
  \left(
    \left.
      \widetilde{\vec{W}}
    \right\|
    \widehat{\vec{W}}
  \right)
  \,
  \diff \tau$
  must be bounded 
\begin{equation}
  \label{eq:244}
  \int_{\tau=0}^{+\infty}
  \mathcal{Y}^{m,n}
  \left(
    \left.
      \widetilde{\vec{W}}
    \right\|
    \widehat{\vec{W}}
  \right)
  \,
  \diff \tau
  <
  +\infty
  .
\end{equation}

Once we have established the boundedness of the integral, we can go back to~\eqref{eq:239} and estimate the right-hand side of~\eqref{eq:239} by its norm. This yields
\begin{equation}
  \label{eq:245}
    \dd{}{t}
  \mathcal{Y}^{m,n}
  \left(
    \left.
      \widetilde{\vec{W}}
    \right\|
    \widehat{\vec{W}}
  \right)
  \leq
  K^{m,n}
  \mathcal{Y}^{m,n}
  \left(
    \left.
      \widetilde{\vec{W}}
    \right\|
    \widehat{\vec{W}}
  \right)
  +
  \mathcal{H}^{m,n}
  \left(
    \left.
      \widetilde{\vec{W}}
    \right\|
    \widehat{\vec{W}}
  \right)
  ,
\end{equation}
and we see that the assumptions of Lemma~\ref{lm:10} are satisfied. Consequently, we can conclude that for $n, m \in (0,1)$, $n>m>\frac{n}{2}$ we have
$
\mathcal{Y}^{m,n}
\left(
  \left.
    \widetilde{\vec{W}}
  \right\|
  \widehat{\vec{W}}
\right)
\xrightarrow{t \to + \infty}
0
$,
which in virtue of the definition of $\mathcal{Y}^{m,n}$, see~\eqref{eq:229}, means that
  \begin{equation}
    \label{eq:248}
    \int_{\Omega}
    \rho
    \cheatvolref
    \widehat{\temp}
    \left[
      \frac{1}{n}
      \left(
        1 + \frac{\widetilde{\temp}}{\widehat{\temp}}
      \right)^{n}
      -
      \frac{1}{m}
      \left(
        1 + \frac{\widetilde{\temp}}{\widehat{\temp}}
      \right)^{m}
      +
      \frac{n-m}{mn}
    \right]
    \,
    \cvolumee
    \xrightarrow{t \to + \infty}
    0
    ,
  \end{equation}
which gives us the decay of temperature perturbations. Now we recall that the decay of the velocity perturbation has been obtained already by a trivial manipulation described in Section~\ref{sec:estim-time-deriv}, hence we know the temperature/velocity perturbation decays as desired. We can summarise our findings as a theorem.

\begin{theorem}[Decay of temperature/velocity perturbations]
  \label{thr:1}
  Let us consider the temperature/velocity perturbation $\widetilde{\temp}$, $\widetilde{\vec{v}}$ to the steady temperature/velocity field $\widehat{\temp}$, $\widehat{\vec{v}}= \vec{0}$. Let us assume that the perturbation is the classical solution to the governing equations~\eqref{eq:governing-equations-perturbation} in domain $\Omega$ with boundary conditions~\eqref{eq:eq:boundary-conditions-perturbation} and \emph{arbitrary} initial conditions~\eqref{eq:initial-conditions-perturbation}, and that this solution exists for all times. Let $m, n \in (0,1)$, $n> m > \frac{n}{2}$, then the temperature perturbation~$\widetilde{\temp}$ satisfies
  \begin{equation}
    \label{eq:249}
    \int_{\Omega}
    \rho
    \cheatvolref
    \widehat{\temp}
    \left[
      \frac{1}{n}
      \left(
        1 + \frac{\widetilde{\temp}}{\widehat{\temp}}
      \right)^{n}
      -
      \frac{1}{m}
      \left(
        1 + \frac{\widetilde{\temp}}{\widehat{\temp}}
      \right)^{m}
      +
      \frac{n-m}{mn}
    \right]
    \,
    \cvolumee
    \xrightarrow{t \to + \infty}
    0
    .
  \end{equation}
  Furthermore, the velocity perturbation $\widetilde{\vec{v}}$ satisfies
  $
  \norm[\sleb{2}{\Omega}]{\widetilde{\vecv}}^2 \leq  \left. \norm[\sleb{2}{\Omega}]{\widetilde{\vecv}}^2 \right|_{t=0} \exponential{- \frac{2 \mu}{C_P \rho}t}
$, where~$C_P$ is the Poincar\'e constant for the domain $\Omega$.
\end{theorem}

The integral~\eqref{eq:249} neither introduces a Lebesgue type norm of the temperature perturbation, nor does it provide an upper bound on a Lebesgue type norm of the temperature perturbation. This might be interpreted as a weakness of the result, since the functional does not measure the distance from the steady state in a proper sense of the word. Let us however see what happens if we rewrite~\eqref{eq:249} in terms of the relative entropy
$
\relentropy = \cheatvolref \ln \left( 1 + \frac{\widetilde{\temp}}{\widehat{\temp}}\right)
$,
see~\eqref{eq:78} and~\eqref{eq:79}. We see that
\begin{equation}
  \label{eq:333}
  \frac{1}{n}
  \left(
    1 + \frac{\widetilde{\temp}}{\widehat{\temp}}
  \right)^{n}
  -
  \frac{1}{m}
  \left(
    1 + \frac{\widetilde{\temp}}{\widehat{\temp}}
  \right)^{m}
  +
  \frac{n-m}{mn}
  =
  \frac{1}{n}
  \exponential{n \frac{\relentropy}{\cheatvolref}}
  -
  \frac{1}{m}
  \exponential{m \frac{\relentropy}{\cheatvolref}}
  +
  \frac{n-m}{mn}.
\end{equation}
Since the temperature~$\temp = \widehat{\temp} + \widetilde{\temp}$ is \emph{bounded from below} uniformly in space and time, see Section~\ref{sec:boundary-conditions}, we know that the relative entropy $\relentropy$ is also bounded from below uniformly in space and time. Concerning the behaviour of the function
$
  f(x) =_{\bydefinition}
  \frac{1}{n}
  \exponential{n x}
  -
  \frac{1}{m}
  \exponential{m x}
  +
  \frac{n-m}{mn}
$,
that appears on the right-hand side of~\eqref{eq:333} we can use Lemma~\ref{lm:9}, and show that for any $l \in \N$, $l\geq 3$ there exists a constant~$L$ such that
\begin{equation}
  \label{eq:352}
  \frac{1}{L}
  \absnorm{
    \frac{\relentropy}{\cheatvolref}
  }^l
  \leq
  \frac{1}{n}
  \exponential{n \frac{\relentropy}{\cheatvolref}}
  -
  \frac{1}{m}
  \exponential{m \frac{\relentropy}{\cheatvolref}}
  +
  \frac{n-m}{mn}
\end{equation}
(See Figure~\ref{fig:relative-entropy-sketch} for a visualisation of this inequality.) The inequality~\eqref{eq:352} then implies that
\begin{equation}
  \label{eq:353}
  0
  \leq
  \frac{1}{L}
  \int_{\Omega}
  \rho \cheatvolref
  \widehat{\temp}
  \absnorm{
    \frac{\relentropy}{\cheatvolref}
  }^l
  \,
  \cvolumee
  \leq
  \int_{\Omega}
  \rho \cheatvolref
  \widehat{\temp}
  \left[
    \frac{1}{n}
    \left(
      1 + \frac{\widetilde{\temp}}{\widehat{\temp}}
    \right)^{n}
    -
    \frac{1}{m}
    \left(
      1 + \frac{\widetilde{\temp}}{\widehat{\temp}}
    \right)^{m}
    +
    \frac{n-m}{mn}
  \right]
  \,
  \cvolumee
  ,
\end{equation}
where the integral on the right-hand side converges to zero as $t \to + \infty$. Consequently, we have proved the following corollary.

\begin{corollary}[Decay of relative entropy]
  \label{crl:1}
  Let the assumptions of Theorem~\ref{thr:1} be fulfilled, and let $l \in [3,+ \infty)$, then
  \begin{equation}
    \label{eq:354}
    \int_{\Omega}
    \rho \cheatvolref
    \widehat{\temp}
    \absnorm{
      \frac{\relentropy}{\cheatvolref}
    }^l
    \,
    \cvolumee
    \xrightarrow{t \to + \infty}
    0
    .
  \end{equation}
\end{corollary}
Note that the convergence to zero in Corrolary~\ref{crl:1} is not uniform with respect to~$l$, since the constant $L$ in~\eqref{eq:352} depends on $l$. The corollary shows that the relative entropy~$\relentropy$ converges to zero in \emph{any} Lebesgue space norm $\sleb{p}{\Omega}$, $p\in [1, +\infty)$, hence the relative entropy unlike the temperature perturbation exhibits the decay in a well established normed space.

\begin{figure}[h]
  \centering
  \subfloat[\label{fig:relative-entropy-sketch-a} Cubic function, $l=3$, $\frac{1}{L}=0.00111937$.]
  {\includegraphics[width=0.46\textwidth]{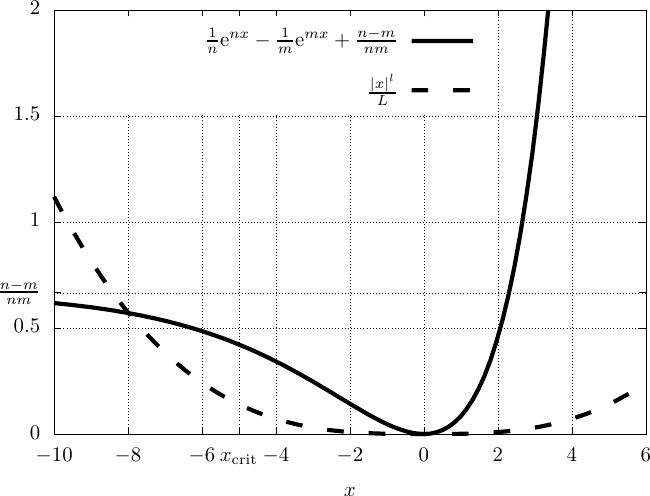}}
  \qquad
  \subfloat[\label{fig:relative-entropy-sketch-b} Quartic function, $l=4$, $\frac{1}{L}=0.000397861$.]
  {\includegraphics[width=0.46\textwidth]{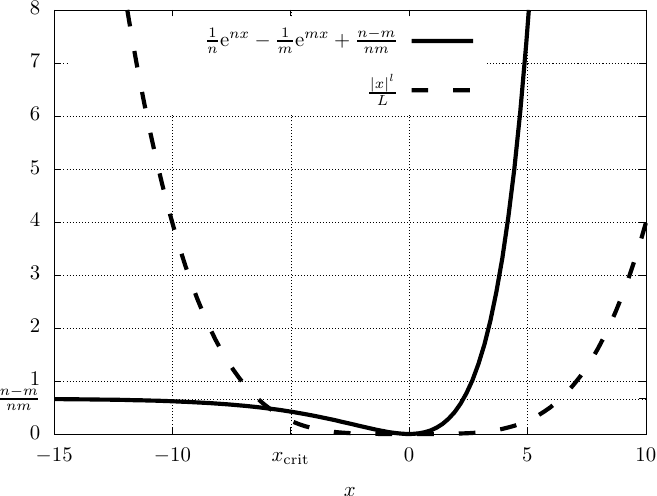}}
  \caption{Visual comparison of functions
    $
    \frac{\absnorm{x}^l}{L}
    $
    and
    $
    \frac{1}{n}
    \exponential{n x}
    -
    \frac{1}{m}
    \exponential{m x}
    +
    \frac{n-m}{mn}
    $ for different exponents~$l$. Parameter values are $n=\frac{1}{2}$, $m=\frac{3}{8}$, $x_{\mathrm{\crit}}=-5$.}
  \label{fig:relative-entropy-sketch}
\end{figure}

\section{Conclusion}
\label{sec:conclusion}

We have investigated the stability of a spatially inhomogeneous non-equilibrium steady state in a thermodynamically open system. Namely we have considered an incompressible heat conducting viscous fluid occupying a vessel with walls kept at spatially non-uniform temperature. The steady state in this system is the rest state (zero velocity) with the temperature field given by the solution of the steady heat equation. This steady state is expected to be stable with respect to arbitrary perturbation.

This trivial fact is difficult to prove using the corresponding evolution equations. The main difficulty is twofold.  First, the system of interest is a thermodynamically \emph{open system}, which means that we have \emph{a priori} no control on the fluxes through the boundary. Second, we consider the full system of governing equations \emph{including the dissipative heating term}. This term is neglected in most of the mathematical works, see~\cite{joseph.dd:stability*1,joseph.dd:stability}, \cite{straughan.b:energy} and references therein. (Albeit these works aim at a different and more complex problem of thermal convection within the context of Oberbeck--Boussinesq approximation, the key difficulties are the same.) This holds also for works that deal with the temperature dependent thermal conductivity or temperature dependent viscosity, see for example~\cite{flavin.jn.rionero.s:benard-problem} and~\cite{diaz.ji.straughan.b:global}. If the stability is considered including the dissipative heating, the available works see~\cite{richardson.ll:nonlinear} and \cite{kagei.y.ruzicka.m.ea:natural} typically lead to conditional results, include \emph{ad hoc} constructed functionals, require high regularity of the velocity field, and are limited to constant transport coefficients.

In our analysis we have assumed that the governing equations---the Navier--Stokes--Fourier equations for an incompressible fluid---have the classical solution that exists for all times. Using the thermodynamically motivated method for the construction of prospective Lyapunov type functionals, see~\cite{bulcek.m.malek.j.ea:thermodynamics}, we have then been able to address all these difficulties. The main theorem, see Theorem~\ref{thr:1}, then states that the temperature and velocity perturbations vanish as time goes to infinity, and that this behaviour holds for any choice of initial perturbation and for any shape of the vessel. This implies the recovery of the spatially inhomogeneous non-equilibrium steady state for \emph{any initial perturbation}.

We note that the only piece of information regarding the behaviour of the dissipative heating term we have used in the stability analysis has been rather weak, which is a very welcome feature. In fact, we have used just the positivity of the term and its integrability in space and time. Consequently, it might be speculated that a similar thermodynamically based analysis is feasible not only for the incompressible Navier--Stokes--Fourier fluid but also for more complex models, provided that one knows a thermodynamic basis for such models\footnote{For a thermodynamic basis for some popular models see for example~\cite{rajagopal.kr.srinivasa.ar:thermodynamic}, \cite{heida.m.malek.j:on}, \cite{malek.j.rajagopal.kr.ea:on}, \cite{hron.j.milos.v.ea:on}, \cite{malek.j.prusa.v.ea:thermodynamics}, \cite{malek.j.rajagopal.kr.ea:derivation}, and~\cite{dostalk.m.prusa.v.ea:on}, where the thermodynamic basis is developed in the form directly applicable in the outlined stability analysis. Note that the basic thermodynamic building blocks are also provided for the models constructed within other thermodynamic frameworks such as the GENERIC framework, see~\cite{pavelka.m.klika.v.ea:multiscale}.}. 
In particular, if the given model properly characterises the degradation of the mechanical energy to the thermal energy, then we would again get the integrability and positivity of the dissipative heating term, and we could essentially follow the same argument as in the present contribution. The only difference would be a different formula for the dissipative heating, and a more involved treatment of the term arising from the convective derivative.   

Once it is clear how to handle the dissipative heating term in the stability analysis of a (physically) simple system, one might speculate about more ambitious applications. In particular, one can again consider a heat conducting viscous fluid occupying a vessel with walls kept at spatially non-uniform temperature. This time one can however consider a fluid with a temperature dependent density, which leads to the Rayleigh--B\'enard convection problem. In such a setting the presence of the \emph{buoyancy term} makes the stability analysis more challenging, since the evolution equation for the thermal and mechanical variables are tightly coupled. The expected result in the Rayleigh--B\'enard setting would be that the classical results, see~\cite{joseph.dd:stability}, are not substantially altered by the presence of the dissipative heating term. Such an analysis is however beyond the scope of this work, and it will be carried out in the future.

\appendix
\section{Derivation of the formulae for the time derivatives of the functionals}
\label{sec:altern-deriv-form}
The derivation presented in the main text can be greatly simplified if we restrict ourselves to formal manipulations. We recall that the main objective is to find the time derivative of the functional
\begin{multline}
  \label{eq:310}
  \mathcal{V}_{\mathrm{meq}}
  \left(
    \left.
      \widetilde{\vec{W}}
    \right\|
    \widehat{\vec{W}}
  \right)
  =
  -
  \int_{\Omega} 
  \rho
  \left[
    \widehat{\temp} \entropy(\widehat{\temp} + \widetilde{\temp})
    -
    \widehat{\temp} \entropy(\widehat{\temp})
    -
    \widehat{\temp} \left. \pd{\entropy}{\temp} \right|_{\temp = \widehat{\temp}} \widetilde{\temp}
    -
    \ienergy(\widehat{\temp} + \widetilde{\temp})
    +
    \ienergy(\widehat{\temp})
    +
    \left. \pd{\ienergy}{\temp} \right|_{\temp = \widehat{\temp}} \widetilde{\temp}
  \right]
  \,
  \cvolumee
  \\
  +
  \int_{\Omega}
  \rho
  \frac{1}{2}
  \absnorm{\widetilde{\vec{v}}}^2
  \,
  \cvolumee
  =
    -
  \int_{\Omega} 
  \rho
  \left[
    \widehat{\temp}
    \left(
      \entropy(\widehat{\temp} + \widetilde{\temp})
      -
      \entropy(\widehat{\temp})
    \right)
    -
    \left(
      \ienergy(\widehat{\temp} + \widetilde{\temp})
      -
      \ienergy(\widehat{\temp})
    \right)
  \right]
  \,
  \cvolumee
  +
  \int_{\Omega}
  \rho
  \frac{1}{2}
  \absnorm{\widetilde{\vec{v}}}^2
  \,
  \cvolumee
  ,
\end{multline}
where we have used the identity $\widehat{\temp} \left. \pd{\entropy}{\temp} \right|_{\temp = \widehat{\temp}} = \left. \pd{\ienergy}{\temp} \right|_{\temp = \widehat{\temp}}$, which holds universally for any material. In the particular case of the fluid described by the Helmholtz free energy in the form~\eqref{eq:free-energy-incompressible-nse}, we get the explicit formula for the functional~$\mathcal{V}_{\mathrm{meq}}$, see Lemma~\ref{lm:2}, equation~\eqref{eq:54}. It is convenient to rewrite~\eqref{eq:54} in terms of the relative entropy,
$
  \relentropy = \cheatvolref \ln \left( 1 + \frac{\widetilde{\temp}}{\widehat{\temp}}\right)
$,
see~\eqref{eq:78} and~\eqref{eq:79}. If we do so, the counterpart of~\eqref{eq:54} reads 
\begin{equation}
  \label{eq:311}
  \mathcal{V}_{\mathrm{meq}}
    \left(
      \left.
        \widetilde{\vec{W}}
      \right\|
      \widehat{\vec{W}}
    \right)
    =_{\bydefinition}
    \int_{\Omega}
    \left[
      \rho
      \cheatvolref
      \widehat{\temp}
      \left[
        \exponential{\frac{\relentropy}{\cheatvolref}}
        -
        \frac{\relentropy}{\cheatvolref}
        -
        1
      \right]
      +
      \frac{1}{2} \rho \absnorm{\widetilde{\vec{v}}}^2
    \right]
    \,
    \cvolumee
    .
\end{equation}
We know that the evolution equation for the entropy of the perturbed state reads
  \begin{equation}
    \label{eq:308}
    \rho
    \dd{\entropy}{t}
    =
    \frac{
      \divergence
      \left(
        \kappa (\widehat{\temp} + \widetilde{\temp}) \nabla \left( \widehat{\temp} + \widetilde{\temp} \right)
      \right)
    }
    {
      \widehat{\temp} + \widetilde{\temp} 
    }
    +
    \frac{\entprodctemp_{\mathrm{mech}}\left(\widehat{\vec{W}} + \widetilde{\vec{W}}\right)}{\widehat{\temp} + \widetilde{\temp} }
    ,
  \end{equation}
  which in our case of constant heat conductivity $\kappa$ and vanishing steady state velocity $\widehat{\vecv} = \vec{0}$ simplifies to
  \begin{equation}
    \label{eq:307}
    \rho
    \pd{\entropy{\left( \widehat{\temp} + \widetilde{\temp} \right)}}{t}
    +
    \rho
    \vectordot{
      \widetilde{\vec{v}}
    }
    {
      \nabla
      \entropy\left( \widehat{\temp} + \widetilde{\temp} \right)
    }
    =
    \frac{
      \divergence
      \left(
        \kapparef
        \nabla
        \left(
          \widehat{\temp}
          \left(
            1 + \frac{\widetilde{\temp}}{\widehat{\temp}}
          \right)
        \right)
      \right)
    }
    {
      \widehat{\temp}
      \left(
        1 + \frac{\widetilde{\temp}}{\widehat{\temp}}
      \right)
    }
    +
    \frac{
      \entprodctemp_{\mathrm{mech}}\left(\widehat{\vec{W}} + \widetilde{\vec{W}}\right)
    }
    {
      \widehat{\temp}
      \left(
        1 + \frac{\widetilde{\temp}}{\widehat{\temp}}
      \right)
    }
    .
  \end{equation}
  The last equation can be again rewritten in terms of the relative entropy,
  \begin{equation}
    \label{eq:312}
    \rho
    \pd{\relentropy}{t}
    =
    \frac{
      \divergence
      \left(
        \kapparef
        \nabla
        \left(
          \widehat{\temp}
          \exponential{\frac{\relentropy}{\cheatvolref}}
        \right)
      \right)
    }
    {
      \widehat{\temp}
      \exponential{\frac{\relentropy}{\cheatvolref}}
    }
    +
    \frac{
      \entprodctemp_{\mathrm{mech}}\left(\widehat{\vec{W}} + \widetilde{\vec{W}}\right)
    }
    {
      \widehat{\temp}
      \exponential{\frac{\relentropy}{\cheatvolref}}
    }
    -
    \rho
    \vectordot{
      \widetilde{\vec{v}}
    }
    {
      \nabla
      \relentropy
    }
    -
    \rho
    \vectordot{
      \widetilde{\vec{v}}
    }
    {
      \nabla
      \widehat{\entropy}
    }
    .
  \end{equation}
  Note that at this point we heavily exploit the fact that in our case the relative entropy is given by a formula that nicely matches with the structure of the heat flux term. Using the evolution equation for the relative entropy, one can derive the evolution equation for any quantity of the form
  $
  \cheatvolref
  \widehat{\temp}
  f
  (
    \exponential{\frac{\relentropy}{\cheatvolref}}
  )
  $
  ,
  where $f$ is a given function, see Lemma~\ref{lm:8}, equation~\eqref{eq:322}. Using the lemma for particular choices of $f$ then allows us to seamlessly derive formulae for the time derivatives of the functionals introduced in Section~\ref{sec:lyap-like-funct} and Section~\ref{sec:reth-devel-lyap}.

    \begin{lemma}[Pointwise evolution equation for functions of the exponential of the relative entropy]
    \label{lm:8}
    Let $\dd{_{\widetilde{\vec{v}}}}{t}= \pd{}{t} + \vectordot{\widetilde{\vec{v}}}{\nabla}$ denote the material time derivative with respect to the perturbed velocity field $\widetilde{\vec{v}}$, and let $f$ denote a given function. The evolution equation for the quantity
$
    \cheatvolref
      \widehat{\temp}
      f
      (
        \exponential{\frac{\relentropy}{\cheatvolref}}
      )
$
    reads
    \begin{multline}
      \label{eq:322}
      \rho
      \dd{_{\widetilde{\vec{v}}}}{t}
      \left[
        \cheatvolref
        \widehat{\temp}
        f
        \left(
          \exponential{\frac{\relentropy}{\cheatvolref}}
        \right)
      \right]
      =
      \divergence
    \left[
      \kapparef
      \nabla
      \left(
        \widehat{\temp}
        f
        \left(
          \exponential{\frac{\relentropy}{\cheatvolref}}
        \right)
      \right)
    \right]
    -
    \kapparef
    \widehat{\temp}
    f^{\prime \prime}
    \left(
      \exponential{\frac{\relentropy}{\cheatvolref}}
    \right)
    \vectordot{
      \nabla
      \exponential{\frac{\relentropy}{\cheatvolref}}
    }
    {
      \nabla
      \exponential{\frac{\relentropy}{\cheatvolref}}
    }
    \\
    +
    f^\prime
      \left(
        \exponential{\frac{\relentropy}{\cheatvolref}}
      \right)
      \entprodctemp_{\mathrm{mech}}\left(\widehat{\vec{W}} + \widetilde{\vec{W}}\right)
    +
    \rho
    \cheatvolref
    \left[
      f
      \left(
        \exponential{\frac{\relentropy}{\cheatvolref}}
      \right)
      -
      f^\prime
      \left(
        \exponential{\frac{\relentropy}{\cheatvolref}}
      \right)
      \exponential{\frac{\relentropy}{\cheatvolref}}
    \right]
      \vectordot{
        \widetilde{\vec{v}}
    }
    {
      \nabla
      \widehat{\temp}
    }
    .
    \end{multline}
  \end{lemma}
  \begin{proof}
    The proof is based on the direct computation. Using the evolution equation for the relative entropy, see~\eqref{eq:312}, we see that
    \begin{multline}
      \label{eq:314}
      \rho
      \pd{}{t}
      \left(
        \widehat{\temp}
        \cheatvolref
        f
        \left(
          \exponential{\frac{\relentropy}{\cheatvolref}}
        \right)
      \right)
      =
      \rho
      \widehat{\temp}
      f^\prime
      \left(
        \exponential{\frac{\relentropy}{\cheatvolref}}
      \right)
      \exponential{\frac{\relentropy}{\cheatvolref}}
      \pd{
        \relentropy
      }
      {
        t
      }
      =
      f^\prime
      \left(
        \exponential{\frac{\relentropy}{\cheatvolref}}
      \right)
      \divergence
      \left[
        \kapparef
        \nabla
        \left(
          \widehat{\temp}
          \exponential{\frac{\relentropy}{\cheatvolref}}
        \right)
      \right]
      \\
      +
      f^\prime
      \left(
        \exponential{\frac{\relentropy}{\cheatvolref}}
      \right)
      \entprodctemp_{\mathrm{mech}}\left(\widehat{\vec{W}} + \widetilde{\vec{W}}\right)
      -
      \rho
      \widehat{\temp}
      f^\prime
      \left(
        \exponential{\frac{\relentropy}{\cheatvolref}}
      \right)
      \exponential{\frac{\relentropy}{\cheatvolref}}
      \vectordot{
        \widetilde{\vec{v}}
      }
      {
        \nabla
        \relentropy
      }
      -
      \rho
      \widehat{\temp}
      f^\prime
      \left(
        \exponential{\frac{\relentropy}{\cheatvolref}}
      \right)
      \exponential{\frac{\relentropy}{\cheatvolref}}
      \vectordot{
        \widetilde{\vec{v}}
      }
      {
        \nabla
        \widehat{\entropy}
      }
      ,
    \end{multline}
    where $f^\prime$ denotes the derivative of function $f$ with respect to its argument. The first term on the right-hand side can be rewritten as a divergence term plus a source term,
  \begin{multline}
    \label{eq:315}
    f^\prime
    \left(
      \exponential{\frac{\relentropy}{\cheatvolref}}
    \right)
    \divergence
    \left[
      \kapparef
      \nabla
      \left(
        \widehat{\temp}
        \exponential{\frac{\relentropy}{\cheatvolref}}
      \right)
    \right]
    =
    \divergence
    \left[
      f^\prime
      \left(
        \exponential{\frac{\relentropy}{\cheatvolref}}
      \right)
      \exponential{\frac{\relentropy}{\cheatvolref}}
      \kapparef
      \nabla
      \widehat{\temp}
      +
      f^\prime
      \left(
        \exponential{\frac{\relentropy}{\cheatvolref}}
      \right)
      \kapparef
      \widehat{\temp}
      \nabla
      \exponential{\frac{\relentropy}{\cheatvolref}}
    \right]
    \\
    -
    \kapparef
    \vectordot{
      \nabla
      \widehat{\temp}
    }
    {
      \exponential{\frac{\relentropy}{\cheatvolref}}
      \nabla
      f^\prime
      \left(
        \exponential{\frac{\relentropy}{\cheatvolref}}
      \right)
    }
    -
    \kapparef
    \widehat{\temp}
    \vectordot{
      \nabla
      \exponential{\frac{\relentropy}{\cheatvolref}}
    }
    {
      \nabla
      f^\prime
      \left(
        \exponential{\frac{\relentropy}{\cheatvolref}}
      \right)
    }
    .
  \end{multline}
  Now we use the identity
$
    \exponential{\frac{\relentropy}{\cheatvolref}}
    \nabla
    f^\prime
    \left(
      \exponential{\frac{\relentropy}{\cheatvolref}}
    \right)
    =
    \nabla
    \left(
      \exponential{\frac{\relentropy}{\cheatvolref}}
      f^\prime
      \left(
        \exponential{\frac{\relentropy}{\cheatvolref}}
      \right)
      -
      f
      \left(
        \exponential{\frac{\relentropy}{\cheatvolref}}
      \right)
    \right)
$,
which allows us to rewrite the corresponding term on right-hand side of~\eqref{eq:315} as
  \begin{equation}
    \label{eq:317}
    \kapparef
    \vectordot{
      \nabla
      \widehat{\temp}
    }
    {
      \exponential{\frac{\relentropy}{\cheatvolref}}
      \nabla
      f^\prime
      \left(
        \exponential{\frac{\relentropy}{\cheatvolref}}
      \right)
    }
    =
    \divergence
    \left[
      \kapparef
      \left(
        \nabla
        \widehat{\temp}
      \right)
    \left(
      \exponential{\frac{\relentropy}{\cheatvolref}}
      f^\prime
      \left(
        \exponential{\frac{\relentropy}{\cheatvolref}}
      \right)
      -
      f
      \left(
        \exponential{\frac{\relentropy}{\cheatvolref}}
      \right)
    \right)
    \right]
    ,
  \end{equation}
  where we have used the fact that $\divergence \left( \kapparef \nabla \widehat{\temp} \right) = 0$. Using~\eqref{eq:317} on the right-hand side of~\eqref{eq:315} reveals that
  \begin{equation}
    \label{eq:318}
        f^\prime
    \left(
      \exponential{\frac{\relentropy}{\cheatvolref}}
    \right)
    \divergence
    \left[
      \kapparef
      \nabla
      \left(
        \widehat{\temp}
        \exponential{\frac{\relentropy}{\cheatvolref}}
      \right)
    \right]
    =
    \divergence
    \left[
      \kapparef
      \nabla
      \left(
        \widehat{\temp}
        f
        \left(
          \exponential{\frac{\relentropy}{\cheatvolref}}
        \right)
      \right)
    \right]
    -
    \kapparef
    \widehat{\temp}
    f^{\prime \prime}
    \left(
      \exponential{\frac{\relentropy}{\cheatvolref}}
    \right)
    \vectordot{
      \nabla
      \exponential{\frac{\relentropy}{\cheatvolref}}
    }
    {
      \nabla
      \exponential{\frac{\relentropy}{\cheatvolref}}
    }
    .
  \end{equation}
  This finishes the manipulation with the first term on the right-hand side of~\eqref{eq:314}. Now we focus on the prospective convective term on the right-hand side of~\eqref{eq:314} . We see that
  \begin{equation}
    \label{eq:319}
    \rho
    \widehat{\temp}
    f^\prime
    \left(
      \exponential{\frac{\relentropy}{\cheatvolref}}
    \right)
    \exponential{\frac{\relentropy}{\cheatvolref}}
    \vectordot{
      \widetilde{\vec{v}}
    }
    {
      \nabla
      \relentropy
    }
    =
    \rho
    \vectordot{
      \widetilde{\vec{v}}
    }
    {
      \nabla
      \left[
        \cheatvolref
        \widehat{\temp}
        f
        \left(
          \exponential{\frac{\relentropy}{\cheatvolref}}
        \right)
      \right]
    }
    -
    \rho
    \cheatvolref
    f
    \left(
      \exponential{\frac{\relentropy}{\cheatvolref}}
    \right)
    \vectordot{
      \widetilde{\vec{v}}
    }
    {
      \nabla
      \widehat{\temp}
    }
    .
  \end{equation}
  Next we observe that $\nabla \widehat{\entropy} = \cheatvolref \frac{\nabla \widehat \temp}{\widehat{\temp}}$, and using~\eqref{eq:318} and~\eqref{eq:319} in~\eqref{eq:314} we finally obtain the evolution equation~\eqref{eq:322}.
\end{proof}

Now we are in the position to exploit Lemma~\ref{lm:8} in the derivation of formulae for the time derivative of the functionals introduced in Lemma~\ref{lm:2} and Lemma~\ref{lm:4}.  
\label{sec:ansatz}
Indeed, if we set
\begin{equation}
  \label{eq:2}
  f(y) =_{\bydefinition} y - \ln y - 1,  
\end{equation}
then $f^\prime(y) = 1 - \frac{1}{y}$, $f^{\prime \prime}(y) = \frac{1}{y^2} $, and using Lemma~\ref{lm:8} we obtain the pointwise evolution equation
  \begin{multline}
    \label{eq:325}
    \rho
    \dd{_{\widetilde{\vec{v}}}}{t}
    \left[
      \cheatvolref
      \widehat{\temp}
      \left(
        \exponential{\frac{\relentropy}{\cheatvolref}}
        -
        \frac{\relentropy}{\cheatvolref}
        -
        1
      \right)
    \right]
    =
    \divergence
    \left[
      \kapparef
      \nabla
      \left(
        \widehat{\temp}
      \left(
        \exponential{\frac{\relentropy}{\cheatvolref}}
        -
        \frac{\relentropy}{\cheatvolref}
        -
        1
      \right)
      \right)
    \right]
    -
    \kapparef
    \widehat{\temp}
    \frac{
      \vectordot{
        \nabla
        \exponential{\frac{\relentropy}{\cheatvolref}}
      }
      {
        \nabla
        \exponential{\frac{\relentropy}{\cheatvolref}}
      }
    }
    {
      \left(
        \exponential{\frac{\relentropy}{\cheatvolref}}
      \right)^2
    }
    \\
    +
    \entprodctemp_{\mathrm{mech}}\left(\widehat{\vec{W}} + \widetilde{\vec{W}}\right)
    -
    \frac{
      \entprodctemp_{\mathrm{mech}}\left(\widehat{\vec{W}} + \widetilde{\vec{W}}\right)
    }
    {
        \exponential{\frac{\relentropy}{\cheatvolref}}
      }
    -
    \rho
    \relentropy
    \vectordot{
        \widetilde{\vec{v}}
    }
    {
      \nabla
      \widehat{\temp}
    }.
  \end{multline}
  The last equation can be upon straightforward manipulations and upon using the fact that $\entprodctemp_{\mathrm{mech}}\left(\widehat{\vec{W}} + \widetilde{\vec{W}}\right) = 2 \mu \tensordot{\widetilde{\gradsym}}{\widetilde{\gradsym}}$ rewritten in terms of the temperature
  \begin{multline}
    \label{eq:329}
    \rho
    \dd{_{\widetilde{\vec{v}}}}{t}
    \left[
      \cheatvolref
      \widehat{\temp}
      \left(
        \frac{\widetilde{\temp}}{\widehat{\temp}}
        -
        \ln
        \left(
          1
          +
          \frac{\widetilde{\temp}}{\widehat{\temp}}
        \right)
      \right)
    \right]
    =
    \divergence
    \left[
      \kapparef
      \nabla
      \left\{
        \widehat{\temp}
      \left(
        \frac{\widetilde{\temp}}{\widehat{\temp}}
        -
        \ln
        \left(
          1
          +
          \frac{\widetilde{\temp}}{\widehat{\temp}}
        \right)
      \right)
      \right\}
    \right]
    \\
    -
    \kapparef
    \widehat{\temp}
      \vectordot{
        \nabla
        \ln
        \left(
          1
          +
          \frac{\widetilde{\temp}}{\widehat{\temp}}
        \right)
      }
      {
        \nabla
        \ln
        \left(
          1
          +
          \frac{\widetilde{\temp}}{\widehat{\temp}}
        \right)
      }
    +
    2 \mu \tensordot{\widetilde{\gradsym}}{\widetilde{\gradsym}}
    -
    \frac{
      2 \mu \tensordot{\widetilde{\gradsym}}{\widetilde{\gradsym}}
    }
    {
      1
      +
      \frac{\widetilde{\temp}}{\widehat{\temp}}
    }
    -
    \rho
    \cheatvolref
    \ln
    \left(
      1
      +
      \frac{\widetilde{\temp}}{\widehat{\temp}}
    \right)
    \vectordot{
        \widetilde{\vec{v}}
    }
    {
      \nabla
      \widehat{\temp}
    }
    .
  \end{multline}
  This is the sought pointwise evolution equation for the thermal part of the integrand in the functional~$\mathcal{V}_{\mathrm{meq}}$. Integrating~\eqref{eq:329} over the (material) domain $\Omega$ leads directly to the result reported in Lemma~\ref{lm:2}. (Recall that the boundary condition $\left. \widetilde{\temp} \right|_{\partial \Omega} = 0$ guarantees that the boundary terms vanish.)

  On the other hand, if we choose $m \in \R$, and if we set
  \begin{equation}
    \label{eq:8}
    f(y) =_{\bydefinition} y - \frac{1}{m} \left(y^m -1\right) - 1,    
  \end{equation}
  then $f^\prime(y) = 1 - y^{m-1}$, $f^{\prime \prime}(y) = -(m-1) y^{m-2} $, and using Lemma~\ref{lm:8} we obtain the pointwise evolution equation
  \begin{multline}
    \label{eq:30}
    \rho
    \dd{_{\widetilde{\vec{v}}}}{t}
    \left[
      \cheatvolref
      \widehat{\temp}
      \left(
        \exponential{\frac{\relentropy}{\cheatvolref}}
        -
        \frac{1}{m} \left( \left(\exponential{\frac{\relentropy}{\cheatvolref}}\right)^m -1 \right)
        -
        1
      \right)
    \right]
    \\
    =
    \divergence
    \left[
      \kapparef
      \nabla
      \left\{
        \widehat{\temp}
      \left(
        \exponential{\frac{\relentropy}{\cheatvolref}}
        -
        \frac{1}{m} \left( \left(\exponential{\frac{\relentropy}{\cheatvolref}}\right)^m -1 \right)
        -
        1
      \right)
      \right\}
    \right]
    \\
    +
    (m-1)
    \kapparef
    \widehat{\temp}
    \frac{
      \vectordot{
        \nabla
        \exponential{\frac{\relentropy}{\cheatvolref}}
      }
      {
        \nabla
        \exponential{\frac{\relentropy}{\cheatvolref}}
      }
    }
    {
      \left(
        \exponential{\frac{\relentropy}{\cheatvolref}}
      \right)^{2-m}
    }
    +
    \entprodctemp_{\mathrm{mech}}\left(\widehat{\vec{W}} + \widetilde{\vec{W}}\right)
    -
    \frac{
      \entprodctemp_{\mathrm{mech}}\left(\widehat{\vec{W}} + \widetilde{\vec{W}}\right)
    }
    {
        \left(\exponential{\frac{\relentropy}{\cheatvolref}}\right)^{1-m}
      }
    +
    \rho
    \cheatvolref
    \frac{m-1}{m}
    \left(
      \left(\exponential{\frac{\relentropy}{\cheatvolref}}\right)^m-1
    \right)
    \vectordot{
        \widetilde{\vec{v}}
    }
    {
      \nabla
      \widehat{\temp}
    }
    .
  \end{multline}
  Using the specific formula for the mechanical dissipation, and rewriting~\eqref{eq:325} in terms of the temperature we get
  \begin{multline}
    \label{eq:331}
    \rho
    \dd{_{\widetilde{\vec{v}}}}{t}
    \left[
      \cheatvolref
      \widehat{\temp}
      \left(
        \frac{\widetilde{\temp}}{\widehat{\temp}}
        -
        \frac{1}{m} \left( \left( 1 + \frac{\widetilde{\temp}}{\widehat{\temp}} \right)^m -1 \right)
      \right)
    \right]
    =
    \divergence
    \left[
      \kapparef
      \nabla
      \left\{
        \widehat{\temp}
      \left(
        \frac{\widetilde{\temp}}{\widehat{\temp}}
        -
        \frac{1}{m} \left( \left( 1 + \frac{\widetilde{\temp}}{\widehat{\temp}} \right)^m -1 \right)
      \right)
      \right\}
    \right]
    \\
    -
    4
    \frac{
      1-m
    }
    {
      m^2
    }
    \kapparef
    \widehat{\temp}
    \vectordot{
      \nabla
      \left[
        \left(1 + \frac{\widetilde{\temp}}{\widehat{\temp}} \right)^{\frac{m}{2}} -1
      \right]
    }
    {
      \nabla
      \left[
        \left(1 + \frac{\widetilde{\temp}}{\widehat{\temp}} \right)^{\frac{m}{2}} -1
      \right]
    }
    \\
    +
    2 \mu \tensordot{\widetilde{\gradsym}}{\widetilde{\gradsym}}
    -
    \frac{
      2 \mu \tensordot{\widetilde{\gradsym}}{\widetilde{\gradsym}}
    }
    {
      \left(1 + \frac{\widetilde{\temp}}{\widehat{\temp}} \right)^{1-m}
      }
    +
    \rho
    \cheatvolref
    \frac{m-1}{m}
    \left(
      \left(1 + \frac{\widetilde{\temp}}{\widehat{\temp}} \right)^m-1
    \right)
    \vectordot{
        \widetilde{\vec{v}}
    }
    {
      \nabla
      \widehat{\temp}
    }
    .
  \end{multline}
  This is the sought pointwise evolution equation for the thermal part of the integrand in the functional~$\mathcal{V}_{\mathrm{meq}}^{\vartemp,\, m}$. Integrating~\eqref{eq:329} over the (material) domain $\Omega$ leads directly to the result reported in Lemma~\ref{lm:4}. 
  
\enlargethispage{1.5em}

\section{Auxiliary tools}
\label{sec:auxiliary-tools-1}

\begin{lemma}
  \label{lm:6}
  Let $m,n \in (0,1)$ and let $n>m$. Let $x \in (-1, +\infty)$ and let us define the function
\begin{equation}
  \label{eq:211}
  f(x,m,n)
  =_{\bydefinition}
    \frac{1}{n}
  \left(
    1 + x
  \right)^{n}
  -
  \frac{1}{m}
  \left(
    1 + x
  \right)^{m}
  +
  \frac{n-m}{mn}
  .
\end{equation}
The function $f(x,m,n)$ is for $x \in (-1, +\infty)$ a non-negative function which vanishes if and only if $x=0$.
\end{lemma}
\begin{proof}
Taking the derivative of~\eqref{eq:211} with respect to $x$ yields
$
  \dd{}{x}
  f(x,m,n)
  =
  \left(
    1 + x
  \right)^{m-1}
  \left[
    \left(
      1 + x
    \right)^{n-m}
    -
    1
  \right]
$.
If $n>m$, then the derivative of $f(x,m,n)$ is positive for $x>0$ and negative for $x<0$. Further the value of $f(x,m,n)$ at $x=0$ is $f(0,m,n)=0$. Consequently, function $f(x,m,n)$ is for all $x \in (-1, +\infty)$ a non-negative function which vanishes if and only if $x=0$.  
\end{proof}

\begin{lemma}
  \label{lm:7}
  Let $m,n \in (0,1)$ and let $n>m> \frac{n}{2}$. Let $x \in (-1, +\infty)$ and let us define the function
\begin{equation}
  \label{eq:213}
  g(x,m,n)
  =_{\bydefinition}
  -
  \left(
    \left[
      \left(
        1
        +
        x
      \right)^{\frac{m}{2}}
      -
      1
    \right]^2
    +
    \left[
      \left(
        1
        +
        x
      \right)^{\frac{n}{2}}
      -
      1
    \right]^2
  \right)
  +
  n
  \left[
    \frac{1}{n}
    \left(
      1 + x
    \right)^{n}
    -
    \frac{1}{m}
    \left(
      1 + x
    \right)^{m}
    +
    \frac{n-m}{mn}
  \right]
  .
\end{equation}
The function $g(x, m, n)$ has the limits
$
  \lim_{x \to -1+} g(x, m ,n) = - 3 + \frac{n}{m} < 0,
$
and
$
  \lim_{x \to + \infty} g(x, m ,n) = - \infty,
$
and the function $g(x, m ,n)$ is non-positive for all $x \in (-1, +\infty)$, and it vanishes if and only if $x=0$.
\end{lemma}
\begin{proof}
Apparently, the function vanishes for $x=0$, and the computation of the limits is straightforward. If we take the derivative of this function with respect to $x$ we get
\begin{equation}
  \label{eq:216}
  \dd{}{x}
  g(x, m ,n)
  =
  \frac{m \left(1+x\right)^{\frac{m}{2}} \left( 1 -  \left(1+x\right)^{\frac{m}{2}} \right) + n \left(1+x\right)^{\frac{n}{2}} \left( 1 -  \left(1+x\right)^{m-\frac{n}{2}} \right)}{1+x},
\end{equation}
and we see that the derivative is negative for $x > 0$, and it is positive for $x < 0$. Consequently, we see that $g(x, m ,n) \leq 0$ for all $x \in (-1, +\infty)$, and that it vanishes if and only if $x=0$.
\end{proof}

\begin{lemma}
  \label{lm:9}
  Let  $m, n \in (0,1)$ and $n> m > \frac{n}{2}$, and let
  $
  f(x) =_{\bydefinition}
    \frac{1}{n}
    \exponential{n x}
    -
    \frac{1}{m}
    \exponential{m x}
    +
    \frac{n-m}{mn}
    $.
  Let $x_{\mathrm{crit}}$ is an arbitrary negative number, and let $l$ is a given natural number, $l \geq 3$, then there exists a constant~$L$ (possibly large) such that
  $
  \frac{\absnorm{x}^l}{L}
  \leq
  f(x)
  $
  holds for all $x \in [x_{\mathrm{crit}}, + \infty)$.
\end{lemma}

\begin{proof}
Using the series expansion for the exponential we can observe that
$
\frac{1}{n}
  \exponential{n x}
  -
  \frac{1}{m}
  \exponential{m x}
  +
  \frac{n-m}{mn}
  =
  \sum_{k=2}^{+\infty}
  \frac{n^{k-1}-m^{k-1}}{k!}
  x^k
$,
and since $n>m$, we see that the inequality
\begin{equation}
  \label{eq:336}
  \frac{n^{l-1}-m^{l-1}}{l!}
  x^l
  \leq
  \frac{1}{n}
  \exponential{n x}
  -
  \frac{1}{m}
  \exponential{m x}
  +
  \frac{n-m}{mn}
\end{equation}
holds for any non-negative $x$ and arbitrary $l\in \N$, $l \geq 2$, and the equality occurs if and only if $x=0$. Hence if we define function $h$ as
\begin{equation}
  \label{eq:342}
  h(x)
  =_{\bydefinition}
  \frac{n^{l-1}-m^{l-1}}{l!}
  \absnorm{x}^l,
\end{equation}
we see that
$
  h(x) < f(x) 
$
for all positive $x$ and $h(x) = f(x)$ for $x=0$.

Moreover, a straightforward calculation also shows that the function $f(x)-h(x)$ has for $l \geq 3$ strict local minimum at zero, hence $h(x)<f(x)$ even in some left neighborhood of zero, that is for some $x \in (-\varepsilon, 0)$. On the other hand function~$h(x)$ is for $x<0$ unbounded, while $f(x)$ is bounded, hence they must intersect,
$
  h(x_{\mathrm{int}}) = f (x_{\mathrm{int}})
$,
at some point $x_{\mathrm{int}}$, where $x_{\mathrm{int}} < 0$. (Among possible intersecting points we chose the one with the smallest magnitude.) If $x_{\mathrm{int}} \leq x_{\mathrm{crit}} $, then we are done, and we have found the desired function. In this case $L$ is given by the formula $\frac{1}{L}= \frac{n^{l-1}-m^{l-1}}{l!}$. 

If $x_{\mathrm{crit}} < x_{\mathrm{int}}$ it remains to flatten the graph of $h(x)$ such that the intersection point is moved sufficiently far to the left. This can be done by the means of the transformation
$
g(x) =_{\bydefinition} \frac{1}{K} h \left( x \right),
$
where $K$ is a sufficiently large positive constant. If we choose $\frac{1}{K} = _{\bydefinition} \frac{h(x_{\mathrm{int}})}{h(x_{\mathrm{crit}})}$, then $\frac{1}{K} < 1$. (Function $h$ is for $x<0$ a strictly decreasing function.) We can observe that $g(x) < f(x)$ for $x \in [x_{\mathrm{int}}, + \infty)$. On the other hand, if $x \in [x_{\mathrm{crit}}, x_{\mathrm{int}})$ then
\begin{equation}
  \label{eq:350}
  g(x)
  =
  \frac{h(x_{\mathrm{int}})}{h(x_{\mathrm{crit}})} h(x)
  <
  h(x_{\mathrm{int}})
  \leq
  f(x_{\mathrm{int}})
  <
  f(x)
  ,
\end{equation}
where we have exploited the fact that $f$ is for $x<0$ a strictly decreasing function. Consequently $g(x) \leq f(x)$ holds for all $x \in [x_{\mathrm{crit}}, + \infty)$ as desired. Moreover, the constant $L$ is given by the formula~$\frac{1}{L} = \frac{f(x_{\mathrm{int}})}{\absnorm{x_{\mathrm{crit}}}^l}$.
\end{proof}

\enlargethispage{1.8em}

\section{Auxiliary tools -- inequalities in function spaces}
\label{sec:auxiliary-tools}
For further reference we recall the standard embedding theorems for Lebesgue and Sobolev spaces, as well as other standard inequalities, for proofs see for example \cite{evans.lc:partial} or \cite{adams.ra.fournier.jjf:sobolev}.

\begin{lemma}[Trivial embedding of Lebesgue spaces]
  \label{lm:trivial-embedding}
  Let $\Omega$ be a domain with a finite volume, and let $p_1, p_2 \in [1, + \infty]$ such that $p_2 \geq p_1$, then
$
    \norm[\sleb{p_1}{\Omega}]{f} \leq \absnorm{\Omega}^{\frac{p_2 - p_1}{p_1 p_2}} \norm[\sleb{p_2}{\Omega}]{f}
    $.
  \end{lemma}

\begin{lemma}[Continuous embedding of Sobolev spaces into Lebesgue spaces]
  \label{lm:sobolev-embedding-continuous}
  Let $\Omega \subset \R^n$ be a domain with $\mathcal{C}^{0,1}$ boundary and let $p \in [1, n)$. Then for all $q \in [1, p^\star]$ where
$
    p^{\star} = _{\bydefinition} \frac{np}{n-p}
$
  there exists a constant $C_{\mathrm{S}}$ that depends on $\Omega$, $p$, $n$ and $q$ such that
$
\norm[\sleb{q}{\Omega}]{f} \leq C_{\mathrm{S}} \norm[\ssob{1}{p}{\Omega}]{f}
$. The constant $C_{\mathrm{S}}$ is referred to as the Sobolev embedding constant.
\end{lemma}

\begin{lemma}[Poincar\'e inequality]
  \label{lm:poincare}
  Let $f \in \ssobzero{1}{p}{\Omega}$, where $p \in [1, +\infty)$ and $\Omega \in {\mathcal C}^{0,1}$. Then there exists a constant $C$ that depends only on $\Omega$ and $p$ such that
$
    \norm[\sleb{p}{\Omega}]{f}
    \leq
    C(\Omega, p)
    \norm[\sleb{p}{\Omega}]{\nabla f}
    $.
  In particular, if $p=2$ then
$
    \norm[\sleb{2}{\Omega}]{f}^2
    \leq
    C_{\mathrm{P}}
    \norm[\sleb{2}{\Omega}]{\nabla f}^2
    $,
  where $C_{\mathrm{P}}$ is referred to as the Poincar\'e constant.
\end{lemma}

\begin{lemma}[H\"older inequality]
  \label{lm:holder}
  Let $\left\{ p_i\right\}_{i=1}^k$ be a sequence of exponents such that $p_i \in [1, +\infty]$ and
  $\sum_{k=1}^{+\infty} \frac{1}{p_k} = 1$. Let us consider functions $\left\{ f_i\right\}_{i=1}^k$, such that $f_i \in \sleb{p_i}{\Omega}$, then
$
    \norm[\sleb{1}{\Omega}]{\prod_{i=1}^k f_i}
    \leq
    \prod_{i=1}^k
    \norm[\sleb{p_i}{\Omega}]{f_i}
$.
\end{lemma}

\begin{lemma}[$\varepsilon$-Young inequality]
  Let $a, b \in \R^+$, $\varepsilon \in \R^+$ and $p,q \in (1, + \infty)$ such that $\frac{1}{p} + \frac{1}{q} = 1$. Then
$
  ab
    \leq
    \varepsilon a^p
    +
    \frac{1}{\left( \varepsilon p \right)^{\frac{q}{p}}q}
    b^q
$.
\end{lemma}

\begin{lemma}[Korn equality]
    Let $\volume$ be a bounded domain with smooth boundary. Let $\vec{v}$ be a smooth vector field that vanishes on the boundary $\left. \vec{v} \right|_{\partial \Omega}=0$, then 
$
    2
  \int_{\Omega}
  {
    \tensordot{\gradsym}{\gradsym}
  }
  \,
  \cvolumee
  =
  \int_{\Omega}
  {
    \tensordot{\gradv}{\gradv}
  }
  \,
  \cvolumee
  +
  \int_{\Omega}
  {
    \left(
      \divergence \vec{v}
    \right)^2
  }
  \,
  \cvolumee
$.
\end{lemma}


%
%

\bibliographystyle{spmpsci-ay}      
\bibliography{vit-prusa}   

\end{document}